\documentclass[reqno]{amsart}
\usepackage{amssymb}
\usepackage[usenames, dvipsnames]{xcolor}
\usepackage{hyperref}
\usepackage{todonotes}

\usepackage{verbatim}

\numberwithin{equation}{section}

\newtheorem{theorem}{Theorem}[section]
\newtheorem{corollary}[theorem]{Corollary}
\newtheorem{lemma}[theorem]{Lemma}
\newtheorem{proposition}[theorem]{Proposition}

\theoremstyle{definition}
\newtheorem{remark}[theorem]{Remark}

\theoremstyle{definition}

\theoremstyle{definition}
\newtheorem{assumption}[theorem]{Assumption}

\makeatletter
\def\dashint{\operatorname%
{\,\,\text{\bf--}\kern-.98em\DOTSI\intop\ilimits@\!\!}}
\makeatother

\def\bR{\mathbb{R}}

\def\cA{\mathcal{A}}
\def\cB{\mathcal{B}}
\def\cC{\mathcal{C}}

\def\cF{\mathcal{F}}

\def\cH{\mathcal{H}}

\def\cM{\mathcal{M}}

\def\cS{\mathcal{S}}

\def\cL{\mathcal{L}}

\begin{document}
\title[Boundedness of non-local operators]{Boundedness of non-local operators with spatially dependent coefficients and $L_p$-estimates for non-local equations }

\author[H. Dong]{Hongjie Dong}
 \address[H. Dong]{Division of Applied Mathematics, Brown University, 182 George Street, Providence, RI 02912, USA}
\email{\href{mailto:Hongjie\_Dong@brown.edu}{\nolinkurl{Hongjie\_Dong@brown.edu}}}

\thanks{H. Dong was partially supported by a Simons fellowship grant no.\,007638, the NSF under agreement DMS-2055244, and the Charles Simonyi Endowment at the Institute of Advanced Study.}

\author[P. Jung]{Pilgyu Jung}
\address[P. Jung]{Department of Mathematics, Korea University, 145 Anam-ro, Seongbuk-gu, Seoul, 02841, Republic of Korea}
\email{\href{mailto:pilgyu_jung@korea.ac.kr}{\nolinkurl{pilgyu_jung@korea.ac.kr}}}

\author[D. Kim]{Doyoon Kim}
\address[D. Kim]{Department of Mathematics, Korea University, 145 Anam-ro, Seongbuk-gu, Seoul, 02841, Republic of Korea}
\email{\href{mailto:doyoon_kim@korea.ac.kr}{\nolinkurl{doyoon_kim@korea.ac.kr}}}

\thanks{P. Jung and D. Kim were supported by the National Research Foundation of Korea (NRF) grant funded by the Korea government (MSIT) (2019R1A2C1084683).}



\begin{abstract}
We prove the boundedness of the non-local operator
\[
\cL^a u(x)=\int_{\bR^d} \left(u(x+y)-u(x)-\chi_\alpha(y)\big(\nabla u(x),y\big)\right) a(x,y)\frac{dy}{|y|^{d+\alpha}}
\]
from $H_{p,w}^\alpha(\bR^d)$ to $L_{p,w}(\bR^d)$ for the whole range of $p \in (1,\infty)$, where $w$ is a Muckenhoupt weight.
The coefficient $a(x,y)$ is bounded, merely measurable in $y$, and H\"{o}lder continuous in $x$ with an arbitrarily small exponent.
We extend the previous results by removing the largeness assumption on $p$ as well as considering weighted spaces with Muckenhoupt weights.
Using the boundedness result, we prove the unique solvability in $L_p$ spaces of the corresponding parabolic  and elliptic non-local equations.
\end{abstract}

\maketitle

\section{Introduction}

We consider non-local parabolic equations of the form
\[
u_t - \cL^a u = f \quad \text{in} \quad \bR^d_T: = (0,T) \times \bR^d
\]
with the non-local operator $\cL^a$ defined by
\begin{equation}\label{05131328}
\cL^a u(t,x)=\int_{\bR^d}\left(u(t,x+y)-u(t,x)-\chi_\alpha(y)\big(\nabla u(t,x),y\big)\right) a(t,x,y)\frac{dy}{|y|^{d+\alpha}},
\end{equation}
where $\alpha\in(0,2)$ and $\chi_\alpha=1_{\alpha>1}+1_{\alpha=1}1_{B_1}.$
Note that the coefficient $a(t,x,y)$ in the operator $\cL^a$ is a function of not only $y \in \bR^d$ but also $(t,x) \in \bR \times \bR^d$.
We also consider non-local elliptic equations of the form
\begin{equation}
							\label{eq1026_01}
\cL^a u = f \quad \text{in} \quad \bR^d,
\end{equation}
where $\cL^a$ is as in \eqref{05131328} with time independent $u(x)$ and $a(x,y)$.
Non-local operators as above appear in the equations describing various phenomena in physics, mathematical finance, biology, and fluid dynamics (see, for instance, \cite{MR2042661, MR2680400}).

The main focus of this paper is the boundedness (i.e., continuity) of the  operator $\cL^a$ from $H_{p,w}^\alpha(\bR^d)$ to $L_{p,w}(\bR^d)$, $1<p<\infty$, under the assumption that the coefficient $a(x,y)$ is measurable in $y$ and continuous in $x$ with an integrability condition on the modulus of continuity (cf. Assumption \ref{05031403} and Remark \ref{rem1026_1}), where $H_{p,w}^\alpha(\bR^d)$ denotes the Bessel potential space with a Muckenhoupt weight $w \in A_p(\bR^d)$.
See Theorem \ref{05140812}.
This result then readily implies the corresponding boundedness for the parabolic case with the coefficient $a(t,x,y)$.
See Corollary \ref{cor0913_01}.
We remark that even in the unweighted case, i.e., $w = 1$, our result is new.

It is well known that, in the study of partial differential equations, a prerequisite is the boundedness of the corresponding differential operators in the considered function spaces.
For instance, in the $L_p$-theory, one usually obtains a priori estimates of solutions and uses the method of continuity along with the unique solvability of a simple equation.
However, all of these steps require the boundedness of the involved operator in an $L_p$ type function space.
In the local case, for instance, if the operator $\cL$ is of the form
\[
\cL u = a^{ij}D_{ij}u,
\]
the boundedness of the operator $\cL$ from $W_p^2(\bR^d)$ to $L_p(\bR^d)$ is straightforward, requiring no regularity assumptions on the coefficients $a^{ij}$, provided that they are bounded measurable functions.
This is also the case when  the $L_p$ space is replaced with a weighted $L_p$ space.
Similarly, for non-local cases, to find a solution $u \in H_p^\alpha(\bR^d)$ of the elliptic equation \eqref{eq1026_01} for a given $f \in L_p(\bR^d)$, it is essential to have
\[
\|\cL^a u\|_{L_p(\bR^d)} \leq N \|u\|_{H_p^\alpha(\bR^d)},
\]
where $H_p^\alpha(\bR^d)$ is a Bessel potential space without weights.
Contrary to the local operator case, the boundedness of non-local operators in $L_p$ spaces is far from being obvious even in the simple case when $a = a(y)$.
Indeed, for non-local operators as in \eqref{05131328}, as far as the authors are aware, there is no result in the literature about the boundedness of the operators if the coefficient $a(t,x,y)$ is merely bounded measurable without any regularity assumptions as a function of $x \in \bR^d$, except for special cases such as the separable coefficient $a(t,x,y) = a_1(t,x)a_2(y)$.

There have been considerable studies about the boundedness of non-local operators as in \eqref{05131328} in $L_p$ spaces along with $L_p$ estimates of solutions to non-local operators.
If the coefficient $a(t,x,y)$ is constant, the operator $\cL^a$ in \eqref{05131328} is, by definition, the fractional derivative of order $\alpha$, whose boundedness is guaranteed for $u \in H_p^\alpha(\bR^d)$.
However, even in the aforementioned simple case, i.e., $a = a(y)$ in the elliptic case and $a = a(t,y)$ in the parabolic case, in early work about the boundedness of $\cL^a$ some conditions were imposed on $a(y)$ or $a(t,y)$ as a function of $y \in \bR^d$.
See \cite{MR2345912, MR1246036}, the results of which can be applied to show boundedness of $\cL^a$ only when $a = a(y)$ or $a=a(t,y)$ is either symmetric, i.e., $a(-y) = a(y)$, or homogeneous of order zero and sufficiently smooth as a function of $y \in \bR^d$.
In fact, the papers \cite{MR2345912, MR1246036} deal with more general forms of non-local operators.
These restrictions on $a(y)$ were removed in \cite{MR2863859} so that $a(y)$ can be bounded measurable, but still needs to be a function of only $y \in \bR^d$.
The boundedness of non-local operators with $x$-dependent kernels was first considered in \cite{MR3145767} for $L_p$ spaces with sufficiently large $p$.
In terms of the operator $\cL^a$ in \eqref{05131328}, the coefficient $a$ can be a function of  $x \in \bR^d$ as well as of  $(t,y) \in \bR^{d+1}$ provided that $a(t,x,y)$ is measurable in $(t,y)$ and H\"{o}lder continuous in $x$.
Using this boundedness result, the authors also proved the existence and uniqueness in Sobolev spaces of solutions to the Cauchy problem  for parabolic integro-differential equations with variable coefficients of the order $\alpha \in (0, 2)$ for sufficiently large $p$.
The main results in \cite{MR3145767} have been applied, for instance, in \cite[Theorem 2.5]{MR4299838} to non-local operators with so-called variable densities, where the investigation was also limited to sufficiently large $p$.

In this paper, we prove the boundedness of the operator $\cL^a$ in \eqref{05131328} when $a=a(t,x,y)$ is measurable in $(t,y)$ and H\"{o}lder continuous in $x$ (see Assumption \ref{05031403} and Remark \ref{rem1026_1}) for all $p \in (1,\infty)$.
That is, we remove the restriction in \cite{MR3145767} that $p$ has to be sufficiently large.
Moreover, we show the boundedness of the operator in weighted $L_p$ spaces with weight $w \in A_p(\bR^d)$.
Nevertheless, whether the boundedness of non-local operators still holds for general $a(t,x,y)$ (for instance, discontinuous or uniformly continuous $a(t,x,y)$ in $x$) remains to be an open problem, for which one can consider various conditions on $a(t,x,y)$.
It is worth mentioning that if H\"{o}lder spaces are considered, the boundedness of the operator $\cL^a$ under optimal conditions is obtained relatively easily by using perturbation arguments. See, for instance, \cite{MR3007688} and \cite{MR3201992}.

As an application of the operator boundedness result, we obtain $L_p$ estimates as in \cite{MR3145767} for all $p \in (1,\infty)$.
See Theorems \ref{05141210} and \ref{thm05241}.
Since our boundedness result is proved in weighted $L_p$ spaces, one can consider weighted $L_p$ estimates for non-local operators.
Regarding results in the weighted setting, see a recent paper \cite{arXiv:2108.11840}.

To prove the boundedness of non-local operators, as a key step, we obtain the boundedness of the operator $\varphi \cL^a$, where $\varphi$ is a cut-off function, from $H_{p,w}^\alpha(\bR^d)$ to $L_{p,w}(\bR^d)$ (for each $t$ in the parabolic case) for all $p \in (1,\infty)$, where $w \in A_p(\bR^d)$.
See Lemma \ref{05021011}.
For the proof of this result, we make an observation that $\cL^a$ with $a=a(y)$ (i.e., $a$ is independent of $x$) is a bounded operator in weighted spaces. See Proposition \ref{05132149}.
We also adapt some arguments from \cite{MR3145767}.

The remaining part of the paper is organized as follows.
In the next section, we introduce some notation and state the main results of the paper.
In Section \ref{aux_sec}, we present some auxiliary results for the proof of the main results.
In Section  \ref{boundedness_sec}, we obtain the boundedness of the operator $\cL^a$ in weighted $L_p$ spaces when the coefficient $a$ is a function of $y$ only.
Finally, we present the proofs of the main results in Section \ref{proof_sec}.

\section{Notation and main results}

For $x\in\bR^d$ and $R\in(0,\infty)$, we denote
\[
B_R(x)=\{y\in \bR^d:|x-y|<R\}, \quad B_R=B_R(0),
\]
where $d$ is a positive integer. As usual, $\cS(\bR^d)$ is the Schwartz function space in $\bR^d$ and $C_0^\infty(\bR^d)$ is the set of all infinitely  differentiable functions with compact support in $\bR^d$.
By $\partial^\alpha$ we mean the fractional Laplacian of order $\alpha$.
That is,
\begin{align*}
\partial^\alpha u (x) = -(-\Delta)^{\alpha/2} u (x) &= c \operatorname{P.V.} \int_{\bR^d} \left( u(x+y)-u(x) \right) \frac{dy}{|y|^{d+\alpha}}
\\
&= \frac{c}{2} \int_{\bR^d} \left( u(x+y) + u(x-y)-2u(x) \right) \frac{dy}{|y|^{d+\alpha}}
\end{align*}
for sufficiently regular $u$ defined on $\bR^d$, where
$$
c = c(d,\alpha) = \frac{\alpha(2-\alpha)\Gamma\left(\frac{d+\alpha}{2}\right)}{\pi^{d/2} 2^{2-\sigma}\Gamma\left(2 - \frac{\alpha}{2}\right)}.
$$

For $1<p<\infty$ and a positive integer $k$, we  set $A_p(\bR^k)$ to be the set of all nonnegative functions $w$ on $\bR^k$ such that
\begin{align*}
&[w]_{A_p(\bR^k)}\\
&:=\sup_{x_0\in \bR^k,R>0}\left(\frac{1}{|B_R|}\int_{B_R(x_0)}w\,dx\right)
\left(\frac{1}{|B_R|}\int_{B_R(x_0)}w^{-1/(p-1)}\,dx\right)^{p-1}<\infty.
\end{align*}
For $w\in A_p(\bR^d)$, by
$L_{p,w}(\bR^d)$ we mean the space of all measurable functions in $\bR^d$ with the norm
\[
\|f\|_{L_{p,w}(\bR^d)}=\left(\int_{\bR^d}|f|^p \, w \, dx\right)^{\frac{1}{p}}.
\]
For $p\in(1,\infty)$, $w\in A_p(\bR^d)$, and $\alpha\in \bR$, recall the definition of the weighted Bessel potential space
	\[
	H_{p,w}^\alpha(\bR^d)=\{u\in L_{p,w}(\bR^d):(1-\Delta)^{\alpha/2}u\in L_{p,w}(\bR^d)\}
	\]
with
\[
\|u\|_{H_{p,w}^\alpha(\bR^d)}=\|(1-\Delta)^{\alpha/2}u\|_{L_{p,w}(\bR^d)}.
\]
As one may expect, we have
\begin{equation}
							\label{eq1022_01}
\|u\|_{H^\alpha_{p,w}(\bR^d)} \approx \|u\|_{L_{p,w}(\bR^d)} + \|\partial^\alpha u\|_{L_{p,w}(\bR^d)}.
\end{equation}
In Lemma \ref{lem10092043} we show  the equivalence of the two norms with constants independent of $\alpha \in (0,2)$.

Set $w(t,x) = w_1(t)w_2(x)$, where $w_1 \in A_q(\bR)$ and $w_2(x) \in A_p(\bR^d)$, $p,q \in (1,\infty)$. Then, we define $L_{p,q,w}(\bR^d_T)$, also denoted by $L_{p,q,w}(T)$, to be the space of integrable functions with the mixed norm
\[
\|u\|_{L_{p,q,w}(T)}=\left(\int_0^T \left(\int_{\bR^d}|u(t,x)|^p w_2(x) \,dx\right)^{q/p} w_1(t) \,dt\right)^{\frac{1}{q}}.
\]
We  write $\cH_{p,q,w}^\alpha(T)$ to indicate  the class of $H_{p,w_2}^\alpha(\bR^d)$-valued functions $u$ on $(0,T)$ equipped with the norm
\[
\|u\|_{\cH_{p,q,w}^\alpha(T)}=\left(\int_0^T\|u(t,\cdot)\|_{H_{p,w_2}^\alpha(\bR^d)}^q w_1(t) \,dt\right)^{\frac{1}{q}}.
\]
As solution spaces for parabolic equations, $\cH_{p,q,w}^{1,\alpha}(T)$ denotes the collection of functions $u\in \cH_{p,q,w}^\alpha(T)$ such that $ u_t\in L_{p,q,w}(T)$.
It is a Banach space with respect to the
norm
\begin{equation*}
\|u\|_{\mathcal{H}_{p,q,w}^{1,\alpha}(T)} = \|u\|_{\cH_{p,q,w}^\alpha(T)}+\|u_t\|_{L_{p,q,w}(T).}
\end{equation*}

As usual, when $w = 1$, in the above notation we simply remove $w$ so that, for instance, $L_p = L_{p,w}$ and $H_p^\alpha = H_{p,w}^{\alpha}$.
We write only $p$ instead of $p,q$ when $p = q$ so that, for instance, $\cH_p^{1,\alpha}(T)$ means $\cH_{p,p,w}^{1,\alpha}(T)$ with $w = 1$.

Throughout the paper, we denote
\[
\nabla_y^\alpha u(x)=u(x+y)-u(x)-\big(\nabla u(x),y\big)(1_{B_1}(y)1_{\alpha=1}+1_{\alpha>1})
\]
for $x, y \in \bR^d$.

We now state our assumption on the coefficient $a(t,x,y)$ in \eqref{05131328}.

\begin{assumption}\label{05031403}
\mbox{}
\begin{enumerate}
\item[(i)] For all $x,y\in\bR^d$ and $t\in(0,T)$, there are positive constants $\delta$ and $K_1$ such that
\[
(2-\alpha)\delta\le a(t,x,y)\le(2-\alpha) K_1.
\]

\item[(ii)] There exist $\beta \in (0,1)$ and a continuous increasing function $\omega(\tau )$, $\tau>0$, such that
\begin{equation*}
|a(t,x_1,y)-a(t,x_2,y)|\leq (2-\alpha)\omega(|x_1-x_2|)
\end{equation*}
for $x_1, x_2,y \in \bR^d$ and $t \in (0,T)$,
and
\begin{equation}
							\label{eq1022_03}	
\int_{|y|\leq 1}\omega(|y|)\frac{dy}{|y|^{d+\beta }}<\infty.
\end{equation}

\item[(iii)] If $\alpha =1$, then for any $(t,x)\in\bR^d_T$ and $0<r<R,$%
\begin{equation*}
\int_{r\le|y|\leq R}ya(t,x,y)\frac{dy}{|y|^{d+\alpha }}=0.
\end{equation*}
\end{enumerate}
\end{assumption}

\begin{remark}
							\label{rem1026_1}
The condition (ii) in Assumption \ref{05031403} is equivalent to a H\"{o}lder continuity condition on $a(t,x,y)$ as a function of $x$.
Since $\omega(\tau)$ is increasing, we see that
\[
\int_{r \leq |y| \leq 2r} \omega(|y|) \frac{dy}{|y|^{d+\beta}} \geq \omega(r) \int_{r \leq |y| \leq 2r} \frac{dy}{|y|^{d+\beta}} \geq N r^{-\beta} \omega(r),
\]
which together with \eqref{eq1022_03} indicates that $a(t,x,y)$ is H\"{o}lder continuous in $x$.
\end{remark}

Our first main result is about  the boundedness of non-local operators.
Let us consider the non-local operator $\cL$ in \eqref{05131328}, where the coefficient $a$ is independent of the time variable.

\begin{theorem}\label{05140812}

 Let $\beta\in(0,1)$, $\alpha\in(0,2)$, $1<p<\infty$, and $w\in A_p(\bR^d)$.
 Suppose that the coefficient $a=a(x,y)$ (i.e., $a$ is independent of the time variable) satisfies Assumption \ref{05031403}.
Then $\cL^a$ is a bounded operator from $H_{p,w}^\alpha(\bR^d)$ to $L_{p,w}(\bR^d).$
More precisely, there exists a constant $N=N(d,p,\alpha,\beta,K_1,\omega, [w]_{A_p})$ such that, for any $u\in H_{p,w}^\alpha(\bR^d)$,
\begin{equation*}
    \|\cL^a u\|_{L_{p,w}(\bR^d)}\le N \|u\|_{H_{p,w}^\alpha(\bR^d)}.
\end{equation*}
Moreover, the constant $N$ can be chosen so that $N = N(d,p,\alpha_0, \alpha_1, \beta, K_1, \omega, [w]_{A_p})$ if $0 < \alpha_0 \leq \alpha \leq \alpha_1 < 1$ and $N = N(d,p,\alpha_0,\beta, K_1, \omega, [w]_{A_p})$ if $1 < \alpha_0 \leq \alpha < 2$, respectively.
\end{theorem}

\begin{remark}
							\label{rem0928_1}
In particular, if $1 < \alpha_0 \leq \alpha < 2$, the constant $N$ is independent of $\alpha$ so that the constant $N$ does not blow up as $\alpha \nearrow 2$.
Regarding this observation, see also, for instance, \cite{MR2494809, MR2863859}.
\end{remark}

When $a = a(t,x,y)$, by using Theorem \ref{05140812} for each $t \in (0,T)$ and then integrating with respect to $t$, we obtain the following corollary for the parabolic operator.

\begin{corollary}
							\label{cor0913_01}
 Let $\beta\in(0,1)$, $\alpha\in(0,2)$, $1<p, q<\infty$, and  $w(t,x) = w_1(t)w_2(x)$, $w_1 \in A_q(\bR)$ and $w_2 \in A_p(\bR^d)$.
 Suppose that the coefficient $a(t,x,y)$ of $\cL^a$ satisfies Assumption \ref{05031403}.
Then the operator $\partial_t-\cL^a$ is continuous from
$\mathcal{H}_{p,q,w}^{1,\alpha}(T)$ to $L_{p,q,w}(T)$.
\end{corollary}

\begin{remark}
As is easy to see, in Corollary \ref{cor0913_01} one can have any nonnegative function $w_1(t)$ instead of $w_1 \in A_q(\bR)$.
\end{remark}

By utilizing the above results, especially, when $w=1$ and $p=q$, as the second set of our main results, we prove the following $L_p$-estimates for parabolic and elliptic equations with spatial non-local operators with H\"older continuous coefficients.
As noted in the introduction, the same type of results are proved in \cite{MR3145767} but for sufficiently large $p$.
That is, we have removed the largeness assumption on $p$ in \cite{MR3145767}.
On the other hand, our boundedness results for operators in the weighted space setting also make it possible to derive weighted versions of the results below. Indeed, such results are proved recently in \cite{arXiv:2108.11840} using Theorem \ref{05140812} and Corollary \ref{cor0913_01} in this paper after detailed preparations for weighted $L_p$-estimates.

We first state the parabolic case.
Note that in the results below we do not pursue whether the constant $N$ stays bounded, for instance, as $\alpha \nearrow 2$ (see Remark \ref{rem0928_1}).
However, it is possible to investigate the dependency of $N$ on $\alpha$ as described in Theorem \ref{05140812}.
See \cite{arXiv:2108.11840} for a priori estimates with more informative dependency of $N$ on $\alpha$.

\begin{theorem}\label{05141210}
Let $\beta\in(0,1)$, $\alpha\in(0,2)$, $1<p<\infty$, and the coefficient $a$ of the operator $\cL^a$ satisfy Assumption \ref{05031403}.
Suppose that $u\in \mathcal{H}_p^{1,\alpha}(T)$ satisfies
 \begin{equation}
							\label{05132019}
\left\{\begin{aligned}
u_t-\cL^a u+\lambda u=f
\quad &\text{in} \quad \bR^d_T,
\\
u(0,x)=0, \quad & x\in \bR^d,
\end{aligned}
\right.
\end{equation}
where $f\in L_p(T).$
 Then, there exists $\lambda_0=\lambda_0(d,p,\alpha,\beta,\delta,K_1,\omega)$ such that  for any $\lambda\ge \lambda_0$,
 \begin{equation}\label{05132016}
         \|u_t\|_{L_p(T)}+\|\partial^\alpha u\|_{L_p(T)}\le N \|f\|_{L_p(T)},
 \end{equation}
 \begin{equation}\label{05132020}
     \|u\|_{L_p(T)}\le N  (T \wedge \lambda^{-1}) \|f\|_{L_p(T)},
 \end{equation}
 where $N=N(d,p,\alpha,\beta,\delta,K_1,\omega)$.
Moreover, for any $f\in L_p(T)$ there exists a unique $u\in \mathcal{H}_p^{1,\alpha}(T)$ satisfying \eqref{05132019}, \eqref{05132016}, and \eqref{05132020}.
\end{theorem}

We also consider parabolic equations with lower-order coefficients.

\begin{theorem}\label{thm05241}
 Let $\beta\in(0,1)$, $\alpha\in(0,2)$, $1<p<\infty$ and $a=a(t,x,y)$ satisfy Assumption \ref{05031403}.
Also let $b^i(t,x)$ and $c(t,x)$ be bounded by $K_2$.
Suppose that $u\in \mathcal{H}_p^{1,\alpha}(T)$ satisfies
 \begin{equation}
							\label{05131515}
\left\{\begin{aligned}
u_t-\cL^a u+b^i D_i u I_{\alpha>1}+cu=f
\quad &\text{in} \quad \bR^d_T,
\\
u(0,x)=0, \quad & x\in \bR^d,
\end{aligned}
\right.
\end{equation}
where $f\in L_p(T)$.
 Then
 \begin{equation}\label{05132004}
     \|u\|_{\mathcal{H}_p^{1,\alpha}(T)}\le N \|f\|_{L_p(T)},
 \end{equation}
 where $N=N(d,p,\alpha,\beta,\delta,K_1,K_2,\omega,T)$. Moreover, for any $f\in L_p(T)$ there exists a unique $u\in \mathcal{H}_p^{1,\alpha}(T)$ satisfying (\ref{05131515}) and (\ref{05132004}).
\end{theorem}

By considering elliptic equations as stationary parabolic equations, we obtain the corresponding results for elliptic equations as follows.

\begin{corollary}
 Let $\beta\in(0,1)$, $\alpha\in(0,2)$, $1<p<\infty$, and the operator $\cL^a$ be independent of the time variable with $a(x,y)$ satisfying  Assumption \ref{05031403}.
Then, there are $\lambda_0=\lambda_0(d,p,\alpha,\beta,\delta,K_1,\omega)$ and $N=N(d,p,\alpha,\beta,\delta,K_1,\omega)$ such that, for any $\lambda\ge \lambda_0$ and
 $u\in H_p^\alpha(\bR^d)$ satisfying
 \begin{equation}\label{05203}
-\cL^a u+\lambda u=f
\quad  \text{in} \quad \bR^d,
\end{equation}
we have
\begin{equation}\label{05204}
      \|\partial^\alpha u\|_{L_p(\bR^d)}+\lambda \|u\|_{L_p(\bR^d)}\le N \|f\|_{L_p(\bR^d)}.
 \end{equation}
Furthermore, for any $f\in L_p(\bR^d)$, there exists a unique $u\in H_p^{\alpha}(\bR^d)$ satisfying (\ref{05203}) and (\ref{05204}).
\end{corollary}
\begin{proof}
Let $\eta$ be a  smooth function in $\bR$ with
\[
\eta(0)=0, \quad \int_0^1 |\eta(t)|\,dt>0.
\]
Set $\eta_n(t)=\eta(t/n)$, $t\in \bR$, and
consider $v_n= u\eta_n$.
By applying Theorem \ref{05141210} with $T=n$ to $v_n$, then letting $n$ go to infinity, we obtain a priori estimate \eqref{05204}.
Then, by the method of continuity and the unique solvability of equations with simple coefficients, for instance, in \cite{MR2863859}, one can finish the proof.
\end{proof}

Throughout  the paper, we may omit $\bR^d$ in $C_0^\infty(\bR^d)$, $\mathcal{S}(\bR^d)$, or $L_p(\bR^d)$  whenever the omission is clear from the context.
Sometimes we use  `$\sup$' to represent the essential supremum.
We write $N(d,\delta,...)$ in the estimates to express that the constant $N$ is determined only by the parameters $d, \delta, \ldots$.
The constant $N$ can differ from line to line.

\section{Auxiliary results}
\label{aux_sec}

In this section, we present some auxiliary results. To make the exposition simple,  we consider $u=u(x)$ and $a=a(x,y)$. In other words, all statements in this section are independent of the time variable.
\begin{lemma}\label{0501_1}
For any $\alpha\in (0,1)$ and $f\in \cS(\bR^d)$,
\[
f(x)=N_0\int_{\bR^d}|z|^{-d+\alpha}\partial^\alpha f(x-z)\,dz,
\]
and
\[
f(x+y)-f(x)=N_0\int_{\bR^d} k^\alpha(z,y)\partial^\alpha f(x-z)\,dz,
\]
where
$$
k^\alpha(z,y)=|z+y|^{-d+\alpha}-|z|^{-d+\alpha}
$$
is an integrable function of $z$ and
\[
N_0=N_0(\alpha,d)=\frac{\Gamma((d-\alpha)/2)}{2^{\alpha}\pi^{d/2}\Gamma(\alpha/2)}.
\]
		  	\end{lemma}

\begin{proof}
The first assertion follows from \begin{equation}\label{eq09251339}
\cF^{-1}[|\xi|^{-\alpha}](x)=N_0|x|^{d-\alpha},
\end{equation}
where $\cF^{-1}$ is the inverse Fourier transform.
See Lemma 2.1 in \cite{MR736974} for the constant $N_0$ and the integrability of $k^\alpha$.	
\end{proof}

\begin{lemma}\label{lem09252211}
Let $\gamma\in (0,2)$, $R \in (0,\infty)$, and $u \in L_{1,\operatorname{loc}}(\bR^d)$.
Then
\begin{equation}\label{eq09252152}
\int_{B_R^c}|u(x+y)|\frac{dy}{|y|^{d+\gamma}}\le N R^{-\gamma} \gamma^{-1}\cM u(x)
\end{equation}
and
\begin{equation*}
\int_{B_R}|u(x+y)|\frac{dy}{|y|^{d-\gamma}}\le N {R^{\gamma}}\gamma^{-1}\cM u(x),
\end{equation*}
where $N$ depends only on $d$.
\end{lemma}

\begin{proof}
Denote $B_j = B_{2^jR}$ for $j=0,1,\dots$, and observe that
\begin{align*}
    &\int_{B_R^c}|u(x+y)|\frac{dy}{|y|^{d+\gamma}} =\sum_{j=0}^\infty \int_{B_{j+1}\setminus B_j}|u(x+y)|\frac{dy}{|y|^{d+\gamma}}
\\
&\le\sum_{j=0}^\infty R^{-(d+\gamma)}2^{-j(d+\gamma)} \int_{B_{j+1}\setminus B_j}|u(x+y)|\,dy \\
&\le N R^{-\gamma} \sum_{j=0}^\infty 2^{-\gamma j} \dashint_{B_{j+1}} |u(x+y)|\,dy\\
    &\le N R^{-\gamma} (1-2^{-\gamma})^{-1}\cM u(x) \le N R^{-\gamma} \gamma^{-1}\cM u(x),
\end{align*}
which shows the first inequality \eqref{eq09252152}.
For the second inequality, we set $B_j = B_{2^{-j}R}$ for $j=0,1,\dots$, and proceed similarly as above with $B_j \setminus B_{j+1}$ in place of $B_{j+1} \setminus B_j$.
\end{proof}

The following lemma is a weighted-$L_p$ norm inequality for multiplier operators in \cite{MR542885}.
From the proofs there, one can see that the constant $N$ depends only on the parameters described below.
As usual, we denote a multi-index by $\gamma = (\gamma_1,\ldots,\gamma_d)$ with $|\gamma| = \gamma_1+\cdots+\gamma_d$, where $\gamma_i$ is a nonnegative integer.

\begin{lemma}\label{10142048}
Let $m$ be a bounded multiplier of an operator $T$ defined on $\cS(\bR^d)$.
If there is a constant $C$ such that for any $|\gamma|\le d$,
\begin{equation}
							\label{eq1022_02}
\sup_{R>0}\left(R^{2|\gamma|-d}\int_{R<|\xi|<2R}|D^{\gamma}m(\xi)|^2\,d\xi\right)^{1/2}\le C,
\end{equation}
then for any $p\in(1,\infty)$ and $w\in A_p(\bR^d)$, there exists a constant $N=N(d,p,C,[w]_{A_p})$  such that
\[
\|Tf\|_{L_{p,w}}\le N\|f\|_{L_{p,w}}
\]
for any $f \in \cS(\bR^d)$.
	\end{lemma}

The next lemma is about the equivalence of the $H_{p,w}^\alpha$ norms stated in \eqref{eq1022_01}.

\begin{lemma}\label{lem10092043}
	Let $\alpha\in[0,\alpha_0]$ and $1<p<\infty$.
Then for any $w\in A_p(\bR^d)$,  there exists a constant $N=N(d,p,\alpha_0,[w]_{A_p})$ such that
	\begin{equation}\label{10092050}
	N^{-1}(\|u\|_{L_{p,w}}+\|\partial^\alpha u\|_{L_{p,w}})\le \|u\|_{H_{p,w}^\alpha}\le N(\|u\|_{L_{p,w}}+\|\partial^\alpha u\|_{L_{p,w}})
	\end{equation}
for any $u \in H_{p,w}^\alpha(\bR^d)$.
\end{lemma}

\begin{proof}
To prove \eqref{10092050}, it suffices to show that, for $u \in \cS(\bR^d)$,
\[
\|(1-\Delta)^{-\alpha/2}u\|_{L_{p,w}} \leq N \|u\|_{L_{p,w}}, \quad \|(-\Delta)^{\alpha/2}(1-\Delta)^{-\alpha/2}u\|_{L_{p,w}} \leq N \|u\|_{L_{p,w}}
\]
for the first inequality,
and
\[
\|(1-\Delta)^{\alpha/2}(1+(-\Delta)^{\alpha/2})^{-1}u\|_{L_{p,w}} \leq N \|u\|_{L_{p,w}}
\]
for the second inequality.
We set
\[
m_1(\xi) = (1 + |\xi|^2)^{\alpha/2}, \quad m_2(\xi) = 1 + |\xi|^{\alpha}.
\]
Then, by differentiating, one can check that the following multiplies
\[
\frac{1}{m_1(\xi)}, \quad \frac{|\xi|^\alpha}{m_1(\xi)}, \quad \frac{m_1(\xi)}{m_2(\xi)}
\]
satisfy the assumptions of Lemma \ref{10142048}, especially, the inequality \eqref{eq1022_02}.
For instance, we use
\[
|D^\gamma m_1|\le N (1+ |\xi|^{2})^{(\alpha-|\gamma|)/2}, \quad |D^\gamma m_2|\le N|\xi|^{\alpha-|\gamma|}+1_{|\gamma|=0},
\]
\[
\left|D^\gamma \left(\frac{1}{m_2(\xi)}\right)\right| \leq N |\xi|^{-|\gamma|} 1_{|\xi|\leq 2} + N |\xi|^{-\alpha-|\gamma|} 1_{|\xi| \geq 1}
\]
to obtain that
\[
\left|D^\gamma\left(\frac{m_1(\xi)}{m_2(\xi)}\right)\right| \leq N |\xi|^{-|\gamma|},
\]
where $N = N(d, \alpha_0, \gamma)$.
In particular, $N$ can be chosen depending on $\alpha_0$ instead of $\alpha$.
The lemma is proved.
\end{proof}

The following lemma is a classical result for the Riesz transforms and the Hilbert transform (see, for instance, \cite{MR2367098}).

\begin{lemma}\label{05191046}
		Let $p\in (1,\infty)$ and $w\in A_p(\bR^d)$. There exists $N=N(d,p,[w]_{A_p})$ such that, for  any $u\in L_{p,w}(\bR^d)$,
		\[
		\|\partial^{-1} \nabla u\|_{L_{p,w}}\le N\|u\|_{L_{p,w}}.
		\]
\end{lemma}

\begin{lemma}\label{05140940}
Let $\alpha\in(0,2)$, $p\in (1,\infty)$, and $w\in A_p(\bR^d)$.
There exists a constant $N=N(d,p,\alpha,[w]_{A_p})$ such that, for any $u\in H_{p,w}^\alpha(\bR^d)$, $\varphi\in C_0^\infty(\bR^d)$,
we have
\begin{align}
        \label{eq092506}
& \int_{\bR^d}\int_{\bR^d} \left(\int_{\bR^d}|u(x+y)-u(x)||\varphi(x+y-z)-\varphi(x-z)|\frac {2-\alpha} {|y|^{d+\alpha}}\,dy\right)^p w(x) \,dx\,dz\notag\\
&\le N\|u\|_{H_{p,w}^{\alpha}}^p \|\varphi\|_{H_p^1}^p.
\end{align}
Moreover, the constant $N$ can be chosen so that $N = N(d,p,\alpha_0, \alpha_1,[w]_{A_p})$ if $0 < \alpha_0 \leq \alpha \leq \alpha_1 < 1$ and $N = N(d,p,[w]_{A_p})$ if $1 \leq \alpha < 2$, respectively.
\end{lemma}

\begin{proof}
We write
\begin{align*}
&\int_{\bR^d}|u(x+y)-u(x)||\varphi(x+y-z)-\varphi(x-z)| \frac{dy}{|y|^{d+\alpha}}\\
&=\left(\int_{B_1}+\int_{B_1^c}\right)\left(\cdots \frac{dy}{|y|^{d+\alpha}}\right).
\end{align*}
By Lemma \ref{lem09252211}, we have
\begin{align}
&\int_{B_1^c}|u(x+y)-u(x)||y|^{-(d+\alpha)}\, dy\notag\\
&\le  \int_{B_1^c}|u(x+y)| |y|^{-(d+\alpha)}\, dy+N\alpha^{-1}|u(x)|
\le N\alpha^{-1}\cM u(x).
\label{eq09252226}
\end{align}
Using the Minkowski inequality in $z$, \eqref{eq09252226}, and the weighted Hardy-Littlewood maximal function theorem, we have
\begin{align}
& \int_{\bR^d}\int_{\bR^d} \left(\int_{B_1^c}|u(x+y)-u(x)||\varphi(x+y-z)-\varphi(x-z)| \frac{dy}{|y|^{d+\alpha}}\right)^p\,dz \, w(x)\,dx\notag\\
&\le N \|\varphi\|_{L_p}^p
\int_{\bR^d} \left(\int_{B_1^c}|u(x+y)-u(x)| \frac{dy}{|y|^{d+\alpha}}\right)^pw(x)\,dx\notag\\
&\le N\alpha^{-p} \|\varphi\|_{L_p}^p
\int_{\bR^d} (\cM u(x))^p w(x)\,dx
\le N\alpha^{-p} \|\varphi\|_{L_p}^p
\|u\|_{L_{p,w}}^p, \label{eq092510}
\end{align}
where $N=N(d,p,[w]_{A_p})$.

Note that for $\alpha<1$, by Lemma \ref{lem09252211},
\begin{align}
                    \label{eq092500}
&\int_{B_1}|u(x+y)-u(x)||y|^{-(d+\alpha-1)}\, dy\notag\\
&\le \int_{B_1}|u(x+y)||y|^{-(d+\alpha-1)}\, dy+N(1-\alpha)^{-1}|u(x)|\notag\\
&\le N(1-\alpha)^{-1}\cM u(x),
\end{align}
where $N=N(d)$.
By using \eqref{eq092500} and the Minkowski inequality, for $\alpha<1$, we have
\begin{align}
&\int_{\bR^d}\int_{\bR^d} \left(\int_{B_1}|u(x+y)-u(x)||\varphi(x+y-z)-\varphi(x-z)| \frac{dy}{|y|^{d+\alpha}}\right)^p dz \, w(x) \, dx\notag\\
&\le \|D\varphi\|_{L_p}^p \int_{\bR^d}\left(\int_{B_1}|u(x+y)-u(x)| \frac{dy}{|y|^{d+\alpha-1}}\right)^p w(x)\,dx\notag\\
&\le N(1-\alpha)^{-p} \|D\varphi\|_{L_p}^p
\int_{\bR^d} (\cM u(x))^pw(x)\,dx\notag\\
&\le N(1-\alpha)^{-p} \|D\varphi\|_{L_p}^p
\|u\|_{L_{p,w}}^p,
                                    \label{eq092502}
\end{align}
where  $N=N(d,p,[w]_{A_p})$.

For $1\le \alpha<2$, we observe that by Lemma \ref{lem09252211} with $R=s$,
\begin{align}
&\int_{B_1}|u(x+y)-u(x)|\frac{dy}{|y|^{d+\alpha-1}} \le \int_0^1\int_{B_1}|Du(x+sy)|\frac{dy}{|y|^{d+\alpha-2}}\,ds\notag
\\
&\le \int_0^1 s^{\alpha-2} \int_{B_s}|Du(x+y)|\frac{dy}{|y|^{d+\alpha-2}}\,ds \le N(2-\alpha)^{-1}\cM (Du)(x),							\label{eq092503}
\end{align}
where the constant $N$ depends only on $d$.
By utilizing \eqref{eq092503} and the Minkowski inequality, we see that
\begin{align}
&\int_{\bR^d}\int_{\bR^d} \left(\int_{B_1}|u(x+y)-u(x)||\varphi(x+y-z)-\varphi(x-z)| \frac{dy}{|y|^{d+\alpha}}\right)^p dz \, w(x) \, dx\notag\\
&\le \|D\varphi\|_{L_p}^p \int_{\bR^d}\left(\int_{B_1}|u(x+y)-u(x)| \frac{dy}{|y|^{d+\alpha-1}}\right)^p w(x)\,dx\notag\\
&\le N(2-\alpha)^{-p} \|D\varphi\|_{L_p}^p
\int_{\bR^d} (\cM Du(x))^p \, w(x)\,dx\notag\\
&\le N(2-\alpha)^{-p} \|D\varphi\|_{L_p}^p
\|Du\|_{L_{p,w}}^p,
                                    \label{eq092508}
\end{align}
where $N=N(d,p,[w]_{A_p})$.
Combining \eqref{eq092510}, \eqref{eq092502}, \eqref{eq092508}, and Lemmas \ref{lem10092043} and \ref{05191046}, we obtain the inequality \eqref{eq092506}.
In particular, when $\alpha \in [1,2)$, by Lemmas \ref{lem10092043} and \ref{05191046},
\[
\|Du\|_{L_{p,w}} \leq \|u\|_{H_{p,w}^1} \leq N \|u\|_{H_{p,w}^\alpha}.
\]

Finally, due to the presence of the term $2-\alpha$ in \eqref{eq092506}, from the  above proof we see that $N$ can be chosen as described in the lemma if $0 < \alpha_0 \leq \alpha \leq \alpha_1<1$ or $1 \leq \alpha < 2$.
\end{proof}

\begin{lemma}\label{05142036}
Let $\alpha \in (0,2)$, $p\in (1,\infty)$, and $w\in A_p(\bR^d)$.
   There exists a constant $N=N(d,p,\alpha, [w]_{A_p})$ such that
   for any $u\in H_{p,w}^\alpha(\bR^d)$,
   \[
   \int_{\bR^d} \left(\int_{|y|\ge 1} \left\lvert \nabla^\alpha_y u(x)\right\rvert \frac{dy}{|y|^{d+\alpha}}\right)^p w(x) \,dx
   \le N\left(\|u\|_{L_{p,w}}^p1_{\alpha\le1}
   +\|Du\|_{L_{p,w}}^p1_{\alpha>1}\right).
   \]
Moreover, the constant $N$ can be chosen so that $N = N(d,p,\alpha_0, [w]_{A_p})$ if $0 < \alpha_0 \leq \alpha \leq 1$ or $1 < \alpha_0 \leq \alpha < 2$. (An upper bound $\alpha_1$ for $\alpha \in (0,1)$ does not appear in this case.)
\end{lemma}

\begin{proof}
 Observe that for $|y|\ge1$,
\[
|\nabla_y^\alpha u(x)|\le \left(\int_0^1 |\nabla u(x+sy)|\,ds+|\nabla u(x)|\right) |y|1_{\alpha>1}
\]
\[+\left(|u(x+y)|+|u(x)|\right)1_{\alpha\le 1},
\]
which implies that
$$
\int_{|y|\ge 1} \left\lvert \nabla^\alpha_y u(x)\right\rvert \frac{dy}{|y|^{d+\alpha}}
\le N\alpha^{-1}\cM u(x)1_{\alpha\le 1}+N(\alpha-1)^{-1}\cM Du(x)1_{\alpha> 1}
$$
by Lemma \ref{lem09252211}.
Then our assertion follows directly from the weighted Hardy-Littlewood maximal function theorem.
 \end{proof}

\begin{lemma}
							\label{05151257}
Let $0<\alpha<2$, $p\in (1,\infty)$, and $w\in A_p(\bR^d)$.
   Then, there is a constant $N=N(d,p,\alpha,[w]_{A_p})$ such that
   for any $u\in H_{p,w}^{2}(\bR^d)$,
\[
\int_{\bR^d} \left(\int_{\bR^d} \left\lvert \nabla^\alpha_y u(x)\right\rvert(2-\alpha) \frac{dy}{|y|^{d+\alpha}}\right)^p w(x) \,dx
   \le
   N \|u\|_{H_{p,w}^{2}}^p.
   \]
Moreover, the constant $N$ can be chosen so that $N = N(d,p,\alpha_0, \alpha_1,[w]_{A_p})$ if $0 < \alpha_0 \leq \alpha \leq \alpha_1 < 1$ and $N = N(d,p,\alpha_0,[w]_{A_p})$ if $1 < \alpha_0 \leq \alpha < 2$, respectively.
\end{lemma}

\begin{proof}
For $\alpha\ge 1$, by using a similar argument as in \eqref{eq092503} {and Lemma \ref{lem09252211}},
we see that there is a constant $N=N(d)$ such that
\begin{align}
    &\int_{B_1} \left\lvert u(x+y)-u(x)-\left(\nabla u(x),y\right)\right\rvert \frac{dy}{|y|^{d+\alpha}}\notag\\
    &\le
    \int_0^1\int_0^1\int_{B_1} s \left\lvert D^2u(x+sty) \right\rvert \frac{dy}{|y|^{d-(2-\alpha)}}\,dt\,ds\notag\\
    &\le N(2-\alpha)^{-1} \cM (D^2u)(x).
    \label{eq092520}
\end{align}
By Lemma \ref{05142036}, \eqref{eq092520}, and the weighted Hardy-Littlewood maximal function theorem, we reach our assertion.
For $\alpha<1$, we proceed similarly upon obtaining
\[
\int_{B_1}|u(x+y) - u(x)| \frac{dy}{|y|^{d+\alpha}} \leq N (1-\alpha)^{-1} \cM (Du)(x).
\]
\end{proof}


\section{Boundedness of non-local operator with \texorpdfstring{$a = a(y)$}{} in \texorpdfstring{$A_p$}{} weighted spaces}
\label{boundedness_sec}

In this section we again consider the time independent case.
If the coefficient $a$ of the non-local operator $\cL$ in \eqref{05131328} is a function of $y$ only, we know that $\cL$ is a bounded operator from $H_p^\alpha$ to $L_p$.
We extend this result so that $\cL$ is a bounded operator from $H^\alpha_p$ to $L_p$ with $A_p$ weights, which is an essential ingredient to the proof of the main theorems.

\begin{proposition}\label{05132149}
	Let $p\in (0,\infty)$ and $\alpha\in(0,2).$
	Suppose that the coefficient $a$ is independent of the $x$ variables, and there exists a constant $K_1$ such that $|a(y)| \leq (2-\alpha)K_1$ for any $y \in \bR^d$.
	When $\alpha=1$, we further assume that, for any $0<r<R$,
\begin{equation}\label{09131350}
\int_{r\le |y|\le R}ya(y)\frac{dy}{|y|^{d+\alpha}}=0.
\end{equation}
	Then, for each $w\in A_p(\bR^d)$,  there is a constant $N=N(d,p, \alpha, [w]_{A_p})$ such that for any $u\in\mathcal{S}(\bR^d)$,
	\begin{equation}
	\label{eq0522_01}
	\|\cL^a u\|_{L_{p,w}}\le N K_1 \|\partial^\alpha u\|_{L_{p,w}}.
	\end{equation}
	Moreover, the constant $N$ can be chosen so that $N = N(d,p,\alpha_0, \alpha_1, [w]_{A_p})$ if $0 < \alpha_0 \leq \alpha \leq \alpha_1 < 1$ and $N = N(d,p,\alpha_0, [w]_{A_p})$ if $1 < \alpha_0 \leq \alpha < 2$, respectively.
Furthermore, if the condition \eqref{09131350} is satisfied for $\alpha \in (0,2)$, the constant $N$ depends only on $d$, $p$, $[w]_{A_p}$, and $\alpha_0$, where $0 < \alpha_0 \leq \alpha < 2$.
\end{proposition}

\begin{remark}
The condition \eqref{09131350} is satisfied, for instance, if $a(y)$ is symmetric, i.e., $a(y) = a(-y)$.
\end{remark}

Before proving Proposition \ref{05132149}, we introduce a well known result  from harmonic analysis.
We  refer the reader to \cite{MR303226, MR358205, MR2182632}.

\begin{lemma}\label{05191147}
	Let $T$ be a singular integral operator of the type
	\[
	T f(x)=\int_{\bR^d} K(x-z)f(z)\,dz.
	\]
	Assume that there are positive constants $\gamma$ and $C_1$ such that
	\[
	|K(x+z)-K(x)|\le C_1\frac{|z|^\gamma}{|x|^{d+\gamma}}, \quad |x|\ge 4 |z|.
	\]
	Also assume that there is a constant $C_2$ such that
\[
\|Tf\|_{L_2}\le C_2 \|f\|_{L_2}
\]
for any $f\in L_2(\bR^d)$.
	Then, for any $w\in A_p(\bR^d)$ and $f\in L_{p,w}(\bR^d)$, there is a constant $N=N(d,p,\gamma,C_1,C_2,[w]_{A_p})$ such that
\[
\|T f\|_{L_{p,w}}\le N \|f\|_{L_{p,w}}.
\]
\end{lemma}

Now, we prove Proposition  {\ref{05132149}}.

\subsection*{Proof of Proposition \ref{05132149} }
To show Proposition \ref{05132149}, we exploit an idea  in \cite{MR3145767}.
We first consider the case when $\alpha\in (0,1)$.
For $\varepsilon\in(0,1)$, set $a_\varepsilon(y)=a(y)1_{\varepsilon\le |y|\le \varepsilon^{-1}}$.
Using Fubini's Theorem and Lemma \ref{0501_1},
\begin{align*}
\cL u&=\int_{\bR^d}\big(u(x+y)-u(x)\big)a(y)\frac{dy}{|y|^{d+\alpha}}\\
&=\lim_{\varepsilon\to 0}N_0\int_{\bR^d}\left(\int_{\bR^d}k^\alpha(z,y)\partial^\alpha u(x-z)\,dz\right) a_\varepsilon(y)\frac{dy}{|y|^{d+\alpha}}\\
&=\lim_{\varepsilon\to 0}N_0\int_{\bR^d}\left(\int_{\bR^d}k^\alpha(z,y)a_\varepsilon(y)\frac{dy}{|y|^{d+\alpha}}\right)\partial^\alpha u(x-z)\,dz.
\end{align*}
For $\varepsilon\in(0,1)$, set
\[
T^\varepsilon f(x)=\int_{\bR^d}k_\varepsilon(x-z) f(z)\,dz,
\]
where
\[
k_\varepsilon(x)=\int_{\bR^d}k^\alpha(x,y)a_\varepsilon(y)\frac{dy}{|y|^{d+\alpha}},
\]
so that in the pointwise sense,
\[
\cL u (x) = \lim_{\varepsilon \to 0} T^\varepsilon \partial^\alpha u (x).
\]

From now on, we check that  $T^\varepsilon$ satisfies the conditions in Lemma \ref{05191147}.
\begin{lemma}\label{05191149}
	Let $0<\alpha_0\le\alpha\le\alpha_1<1$. Assume that $|a(y)|\le 1$ for all $y\in \bR^d$. Then,
	there is a constant $N=N(d,\alpha_0,\alpha_1)$ such that for any $f \in L_2(\bR^d)$,
	\begin{equation*} 
	\|T^\varepsilon f\|_{L_2}\le N\|f\|_{L_2}.
	\end{equation*}
\end{lemma}

\begin{proof}

Since
\[
\cF [T^\varepsilon f] = \cF[k_\varepsilon] \cF[f],
\]
we only show that $\cF[k_\varepsilon]$ is bounded.
Bearing \eqref{eq09251339} in mind,
 we see that, by denoting $\hat{\xi}=\xi/|\xi|$,
\begin{align*}
&|\cF[k_\varepsilon](\xi)|
=\left|\int_{\bR^d} \cF[k^\alpha(\cdot,y)](\xi)a_\varepsilon(y)\frac{dy}{|y|^{d+\alpha}}\right|\\
&\le N|\xi|^{-\alpha}\int_{\bR^d} |e^{i(\xi,y) }-1|\frac{dy}{|y|^{d+\alpha}}
=
N\int_{\bR^d} |e^{i(\hat{\xi},y) }-1|\frac{dy}{|y|^{d+\alpha}}\\
&\le N\int_{B_1} \frac{dy}{|y|^{d+\alpha-1}}
+
N\int_{B_1^c} \frac{dy}{|y|^{d+\alpha}}
\le N(d,\alpha_0,\alpha_1).
\end{align*}
The lemma is proved.
\end{proof}

\begin{lemma}\label{0501_5}
	Let $0<\alpha_0\le\alpha\le\alpha_1<1$.
Suppose that for all $y\in\bR^d$, $|a(y)|\le 1$.
Then, for all $|x|>4|z|$,
	\begin{equation}\label{05141741}
	|k_\varepsilon(x+z)-k_\varepsilon(x)|\le N \frac{|z|^\gamma}{|x|^{d+\gamma}},
	\end{equation}
where $N=N(d,\alpha_0,\alpha_1)$ and $\gamma=\gamma(\alpha_0,\alpha_1)$.
\end{lemma}
\begin{proof}
	In fact, the lemma follows from the proof of \cite[Lemma 14]{MR3145767}, where the authors prove, instead of \eqref{eq0522_01}, the boundedness of $\cL^a$ in $L_p$ spaces {\em without weights} by obtaining
	\[
	\int_{|x|> 4|z|} |k_\varepsilon(x+z) - k_\varepsilon(x)| \, dx \leq C, \quad z \in \bR^d
	\]
	instead of \eqref{05141741}.
	However, the proof there also shows the inequality \eqref{05141741}.
	For the reader's convenience, here we present a slightly simplified version of the proof of \cite[Lemma 14]{MR3145767} dedicated to verify the inequality \eqref{05141741}.
	
	Fix $\varepsilon\in(0,1)$.
	For notational simplicity, we replace $a_\varepsilon$, $k_\varepsilon$ with $a$, $k$, respectively.
	For any $r>0$,
\begin{align*}
&k(rx)=\int_{\bR^d} k^\alpha(rx,y)a(y)\frac{dy}{|y|^{d+\alpha}}\\
&	=\int_{\bR^d}\left( |rx+y|^{-d+\alpha}-|rx|^{-d+\alpha}\right)a(y)\frac{dy}{|y|^{d+\alpha}}\\
&	= r^{-d+\alpha}\int_{\bR^d} \left(|x+\frac{y}{r}|^{-d+\alpha}-|x|^{-d+\alpha}\right)a(y)\frac{dy}{|y|^{d+\alpha}}\\
&=r^{-d}\int_{\bR^d}\left(|x+y|^{-d+\alpha}-|x|^{-d+\alpha}\right)a(ry)\frac{dy}{|y|^{d+\alpha}}\\
&	=:r^{-d}k(x,r).
\end{align*}
	Hence, for $|x|>4|z|$,  by writing
\[
k(x+z) - k(x) =k\left(|x|\frac{x+z}{|x|}\right) - k\left(|x|\frac{x}{|x|}\right)
\]
and using the above $k(\cdot,\cdot)$, we have
\begin{align*}
 &|k(x+z)-k(x)|=|x|^{-d} \left(k\left(\frac{x+z}{|x|},|x|\right)-k\left(\frac{x}{|x|},|x|\right)\right)\\
&	\le
	|x|^{-d}\int \left \lvert \left \lvert\hat x+\frac{\hat z}{r}+y\right\rvert^{-d+\alpha}-\left\lvert\hat x +\frac{\hat z}{r}\right\rvert^{-d+\alpha}-|\hat x+y|^{-d+\alpha}+|\hat{x}|^{-d+\alpha}\right \rvert \frac{dy}{|y|^{d+\alpha}},
 \end{align*}
where $r=|x|/|z| (>4)$, $\hat x=x/|x|$ and $\hat z=z/|z|.$
	By using an orthogonal transformation, we may assume that $\hat x=e=(1,0,\dots,0)\in\bR^d.$
	We split
\begin{align*}
	I&=\int \left \lvert \left \lvert e+\frac{\hat z}{r}+y\right \rvert ^{-d+\alpha}-\left \lvert e +\frac{\hat z}{r}\right\rvert^{-d+\alpha}-|e+y|^{-d+\alpha}+|e|^{-d+\alpha}\right \rvert \frac{dy}{|y|^{d+\alpha}}\\
&	=\int_{|y|\le \frac{1}{r}}\cdots+\int_{\frac{1}{r}\le |y|, \frac{1}{2}\le|y+e| }\cdots +\int_{|y+e|\le \frac{2}{r}}\cdots +\int_{\frac{2}{r}\le |y+e|\le \frac{1}{2}} \cdots\\
&	=: I_1+I_2+I_3+I_4.
\end{align*}
	For any $a,b\in \bR^d$ satisfying $|a+tb|>0$ for any $t\in [0,1]$, it follows that
	\begin{equation}\label{0502_1}
	\begin{aligned}
	\left \lvert  |a+b|^{-d+\alpha}-|a|^{-d+\alpha}\right \rvert &\le N|b|\int_0^1|a+tb|^{-d+\alpha-1}\,dt
	\\
	&\le N|b| (\underset{0\le t\le 1}{\inf}|a+tb|)^{-d+\alpha-1}.
	\end{aligned}
	\end{equation}
{\em Estimate of $I_1$.}
	For any $0\le t\le 1$ and $|y|\le 1/r$,
	$$
	|e+ty|, \quad |e+\frac{\hat z}{r}+ty|\ge \frac{1}{2}.
	$$
	Thus, bearing (\ref{0502_1}) in mind,
\begin{align*}
I_1&\le \int_{|y|\le \frac{1}{r}}\left \lvert |e+\frac{\hat z}{r}+y|^{-d+\alpha}-|e +\frac{\hat z}{r}|^{-d+\alpha}\right\rvert+\left\lvert|e+y|^{-d+\alpha}-|e|^{-d+\alpha}\right \rvert \frac{dy}{|y|^{d+\alpha}}\\
	&\le N\int_{|y|\le \frac{1}{r}} \frac{dy}{|y|^{d+\alpha-1}}\le N r^{\alpha-1}.
\end{align*}
{\em Estimate of $I_2$.}
	For any $0\le t\le 1$ and $|y+e|\ge \frac{1}{2}$,
	$$
	|e+t\frac{\hat z}{r}+y|, \quad |e+t\frac{\hat z}{r}|\ge \frac{1}{4}.
	$$
	By (\ref{0502_1}), $I_2$ is bounded by
\begin{align*}
	&
	\int_{\frac{1}{r} \le |y|,\frac{1}{2}\le |e+y|}\left \lvert |e+\frac{\hat z}{r}+y|^{-d+\alpha}-|e +y|^{-d+\alpha}\right\rvert+\left\lvert|e+\frac{\hat z}{r}|^{-d+\alpha}-|e|^{-d+\alpha}\right \rvert \frac{dy}{|y|^{d+\alpha}}\\
&	\le
	\frac{N}{r}\int_{\frac{1}{r}\le |y|}\frac{dy}{|y|^{d+\alpha}}
	\le N r^{\alpha-1}.
\end{align*}
{\em Estimate of $I_3$.}
For $|y+e|\le \frac{2}{r}$, we see that
	$$
	|y|\ge \frac{1}{2}, \quad
	|e+\frac{\hat z}{r}+y|\le \frac{3}{r}.
	$$
	Therefore, using $|e+\hat{z}/r|\ge 1/2$,
\begin{align*}
	I_3&\le N\int_{|e+\frac{\hat z}{r}+y|\le \frac{3}{r}}\left( |e+\frac{\hat z}{r}+y|^{-d+\alpha}+1\right)\,dy
	+N\int_{|e+y|\le \frac{2}{r}}\left(|e+y|^{-d+\alpha}+1\right)\,dy\\
&	\le N\int_{|y|\le \frac{3}{r}}|y|^{-d+\alpha}\,dy
	\le N r^{-\alpha}.
\end{align*}
{\em	Estimate of $I_4$.}
	For $\frac{2}{r}\le |y+e|\le \frac{1}{2}$, $t\in[0,1]$, we have
	\[
	\frac{1}{r}\le |e+t\frac{\hat z}{r}+y|\le 1, \quad \frac{1}{2}\le |y|, \quad
	|e+t\frac{\hat{z}}{r}|\ge \frac{1}{2}.
	\]
	By (\ref{0502_1}), we arrive at
\begin{align*}
	I_4
&	\le
	\int_{\frac{2}{r}\le |e+y|\le \frac{1}{2}}\left \lvert |e+\frac{\hat z}{r}+y|^{-d+\alpha}-|e +y|^{-d+\alpha}\right\rvert+\left\lvert|e+\frac{\hat z}{r}|^{-d+\alpha}-|e|^{-d+\alpha}\right \rvert \frac{dy}{|y|^{d+\alpha}}\\
&	\le N r^{-1}\int_0^1 \int_{\frac{1}{r}\le |e+t\frac{\hat z}{r}+y|\le 1}|e+t\frac{\hat z}{r}+y|^{-d+\alpha-1}\,dy\,dt\\
&\quad 	+N r^{-1}\int_0^1\int_{\frac{2}{r}\le |e+y|\le 1/2}|e+t\frac{\hat z}{r}|^{-d+\alpha-1} \,dy\,dt\\
&	\le N r^{-1}\int_{\frac{1}{r}\le |y|\le 1}|y|^{-d+\alpha-1}\,dy+N r^{-1}\le N r^{-\alpha}.
\end{align*}
Finally, by taking $\gamma = \min \{ 1 - \alpha_1, \alpha_0\}$, we have
	\[
	I\le N r^{ -\gamma}.
	\]
The lemma is proved.
\end{proof}

Let us continue the proof of Proposition \ref{05132149}.
Clearly, we can assume that $K_1 = 1$.
By combining Lemmas \ref{05191147},  \ref{05191149}, and \ref{0501_5}, we see that the assertion in Proposition \ref{05132149}  holds when $\alpha\in(0,1)$.

To complete the proof of Proposition \ref{05132149}, we proceed as follows.
First, we prove again the inequality \eqref{eq0522_01} for $\alpha \in [3/4,1)$ to present the exact dependency of the constant $N$ on $\alpha$ as $\alpha \nearrow 1$.
Then, using this result, we prove the inequality \eqref{eq0522_01} for $\alpha \in (1,2)$.
In particular, we show that the constant $N$ can be chosen independent of $\alpha$ if $1 < \alpha_0 \leq \alpha <2$.
Finally, we prove the case $\alpha = 1$.
In the steps below, we use the idea in the proof of \cite[Lemma 15]{MR3145767}.

As above, we set
$a_\varepsilon(y)=a(y)I_{\varepsilon<|y|<\varepsilon^{-1}}$ and denote
\[
\cL_\varepsilon^a u (x) = \int_{\bR^d} \nabla_y^\alpha u(x) a_\varepsilon(y)\frac{dy}{|y|^{d+\alpha}}.
\]
Then we have
\[
\cL u(x)=\lim_{\varepsilon\to 0} \cL_\varepsilon^a u (x).
\]
Thus, to prove \eqref{eq0522_01} it suffices to show that
\begin{equation}
							\label{eq0922_03}
\|\cL_\varepsilon^a u \|_{L_{p,w}} \leq N \|\partial^\alpha u\|_{L_{p,w}},
\end{equation}
where $N = N(d,p,\alpha, [w]_{A_p})$, but independent of $\varepsilon$.
Let $\alpha \in [3/4,1)$.
By Lemma \ref{0501_1}, we see that
\begin{align}
							\label{eq0922_02}
&\cL_\varepsilon^a u(x) = N_0 \int_{\bR^d}\int_{\bR^d} k^{1/2}(z,y)\partial^{1/2}u(x-z)\,dz \, a_\varepsilon(y)\frac{dy}{|y|^{d+\alpha}}\notag
\\
&=N_0\int_{\bR^d}\int_{\bR^d} k^{1/2}(z,y)(\partial^{1/2}u(x-z)-\partial^{1/2}u(x))\,dz \, a_\varepsilon(y)\frac{dy}{|y|^{d+\alpha}}\notag
\\
&=N_0\int_{\bR^d}\left(\int_{\bR^d} k^{1/2}(z,y) a_\varepsilon(y)\frac{dy}{|y|^{d+\alpha}}\right)(\partial^{1/2}u(x-z)-\partial^{1/2}u(x))\,dz,
\end{align}
where $N_0 = N_0(d)$ and the second equality is due to
\[
\int_{\bR^d} k^{1/2}(z,y) \, dz  = 0.
\]
Note that, if we set
\[
m_\varepsilon(z) := \int_{\bR^d} \left( \left| \frac{z}{|z|} + y\right|^{-d+1/2} - 1 \right) a_\varepsilon(|z|y) \frac{dy}{|y|^{d+\alpha}},
\]
then, for $z \neq 0$,
\[
\int_{\bR^d} k^{1/2}(z,y) a_\varepsilon(y)\frac{dy}{|y|^{d+\alpha}} = \frac{1}{|z|^{d-1/2+\alpha}} m_\varepsilon(z).
\]
From \eqref{eq0922_02}, we get that
\begin{equation}\label{eq0922_07}
 \cL_\varepsilon^a u(x) = N_0 \int_{\bR^d} (\partial^{1/2}u(x-z)-\partial^{1/2}u(x)) m_\varepsilon(z) \, \frac{dz}{|z|^{d-1/2+\alpha}}.
\end{equation}
If we have
\begin{equation}\label{eq0922_01}
|m_{\varepsilon}(z)| \leq (1-\alpha)^{-1}N
\end{equation}
for any $z\neq 0$, where $N$ depends only on $d$, but independent of $\varepsilon$,
then the representation (\ref{eq0922_07}) along with Proposition \ref{05132149} for order $\alpha-1/2$ proves \eqref{eq0922_03}.
Thus, it remains to show \eqref{eq0922_01}.
Upon denoting $\hat{z}=z/|z|$, we write
\begin{align*}
m_\varepsilon(z)&=\int_{|y|\le 1/2}\left( \frac{1}{|\hat{z}+y|^{d-1/2}}-1\right)a_\varepsilon(|z|y)\frac{dy}{|y|^{d+\alpha}}\\
&\quad +\int_{|y|>1/2}\left(\frac{1}{|\hat{z}+y|^{d-1/2}}-1\right)a_\varepsilon(|z|y)\frac{dy}{|y|^{d+\alpha}}=:I_\varepsilon+J_\varepsilon.
\end{align*}
Set
\[
f(t)=\frac{1}{|\hat{z}+ty|^{d-1/2}}, \quad t \in [0,1].
\]
Then, since $|\hat{z}+ty|\ge 1/2$ for $|y|\le 1/2$,
	\[\left|\frac{1}{|\hat{z}+y|^{d-1/2}}-1\right|=|f(1)-f(0)|
\le \sup_{t\in[0,1]}|f'(t)|\le N|y|.
	\]
	Therefore,
\[
|I_\varepsilon| \le N\int_{|y|\le 1/2} |a_\varepsilon(|z|y)|\frac{dy}{|y|^{d+\alpha-1}}
\le (1-\alpha)^{-1}N,
\]
where $N = N(d)$.
	To estimate $J_\varepsilon$, we split
	\[
	J_\varepsilon
	=\int_{|y|>1/2, |\hat{z}+y|\ge 1/4}\left(\frac{1}{|\hat{z}+y|^{d-1/2}}-1\right)a_\varepsilon(|z|y)\frac{dy}{|y|^{d+\alpha}}
	\]
	\[
	+\int_{ |\hat{z}+y|\le 1/4}\left(\frac{1}{|\hat{z}+y|^{d-1/2}}-1\right)a_\varepsilon(|z|y)\frac{dy}{|y|^{d+\alpha}},
	\]
where both integrals are bounded by a constant depending only on $d$ because $\alpha \geq 3/4$,
\[
\left|(|\hat{z}+y|^{-d+1/2}-1)a_\varepsilon(|z|y)\right| \leq N
\]
in the first integral, and $|y| \geq 1/2$ in the second integral.
Hence, we have obtained \eqref{eq0922_01}.
Since $1/4\le\alpha-1/2<1/2$, it follows that (\ref{eq0922_03}) holds with
\begin{equation}\label{eq0922_05}
N=(1-\alpha)^{-1}N(d,p, \alpha-1/2, [w]_{A_p})=(1-\alpha)^{-1}N(d,p,[w]_{A_p}).
\end{equation}

For $\alpha\in (1,2)$, $\cL^a u$ can be written in the form
\[
\sum_{i=1}^d \int_{\bR^d} \left(D_i u(x+y)-D_iu(x)\right)a_i(y)\frac{dy}{|y|^{d+\alpha-1}} =\sum_{i=1}^d \cL^{a_i}(D_i u)(x),
\]
where
\[
a_i(y)=\frac{y^i}{|y|}\int_0^1 a(y/s)s^{-1+\alpha}\,dx.
\]
Note that the order of $\cL^{a_i}$ is $\alpha-1$.
By Proposition \ref{05132149} for $\alpha \in (0,1)$ just proved above and Lemma \ref{05191046},
\[
\|\cL^a u\|_{L_{p,w}}\le \sum_{i=1}^d \|\cL^{a_i}D_i u\|_{L_{p,w}}
\le N \| \partial^{\alpha-1}\nabla u\|_{L_{p,w}}\le N \|\partial^\alpha u\|_{L_{p,w}}.
\]
From $|a|\le (2-\alpha)$  and (\ref{eq0922_05}),
we see that $N = N(d,p,\alpha_0,[w]_{A_p})$ if $1 < \alpha_0 \leq \alpha < 2$ with no blow-up as $\alpha \nearrow 2$.

Finally, we prove \eqref{eq0922_03} for $\alpha \in (3/4,5/4)$ with $N = N(d,p,[w]_{A_p})$ under the additional assumption that $a(y)$ satisfies \eqref{09131350}.
Then, this proves \eqref{eq0522_01} when $\alpha = 1$.
In addition, this observation combined with the results proved above justifies the last assertion in the proposition that $N$ can be chosen depending on $\alpha_0$, but independent of $\alpha \in (0,2)$ if $0 < \alpha_0 \leq \alpha < 2$.
In particular, the constant $N$ does not blow up as $\alpha \to 1$ if \eqref{09131350} is satisfied.
 Using $(\ref{09131350})$, we have
\[
\cL_\varepsilon^a u(x) = N_0\int_{\bR^d} (\partial^{1/2}u(x-z)-\partial^{1/2}u(x)) m_\varepsilon(z) \, \frac{dz}{|z|^{d+\alpha-1/2}},
\]
where
\[
m_\varepsilon(z)=\int_{|y|\le 1/2}\left( \frac{1}{|\hat{z}+y|^{d-1/2}}-1-(-d+1/2)(\hat{z},y)\right)a_\varepsilon(|z|y)\frac{dy}{|y|^{d+\alpha}}
	\]
	\[+\int_{|y|>1/2}\left(\frac{1}{|\hat{z}+y|^{d-1/2}}-1\right)a_\varepsilon(|z|y)\frac{dy}{|y|^{d+\alpha}}=:I_\varepsilon+J_\varepsilon.
\]
Then, since $|\hat{z}+ty|\ge 1/2$ for $|y|\le 1/2$,
\begin{align*}
&\left|\frac{1}{|\hat{z}+y|^{d-1/2}}-1-(-d+1/2)(\hat{z},y)\right|=|f(1)-f(0)-f'(0)|\\
&\le \frac{1}{2}\sup_{t\in[0,1]}|f''(t)|\le N|y|^2,
\end{align*}
we have
\[
|I_\varepsilon| \le N\int_{|y|\le 1/2} |a_\varepsilon(|z|y)|\frac{dy}{|y|^{d+\alpha-2}}
\le N,
\]
where $N = N(d)$. Note that the estimate of $J_\varepsilon$ still depends only on $d$.
By  using the proposition for $1/4< \alpha-1/2< 3/4$ proved above,  we have (\ref{eq0922_03}) with
\[N=N(d,p,\alpha-1/2, [w]_{A_p})=N(d,p,[w]_{A_p}).\]
This finishes the proof of Proposition \ref{05132149}.\qed

\section{Proof of main theorems}
\label{proof_sec}

We start with a key lemma.
\begin{lemma}\label{05021011}
Let $\beta\in (0,1)$, $\alpha\in (0,2)$, $p\in (1,\infty)$, $w\in A_p(\bR^d)$,  $\varepsilon\in (0,1)$, and $\varphi\in C_0^\infty(B_\varepsilon)$.
Also let $a(x,y)$ be a bounded measurable function.
Suppose that there is a continuous increasing function $\omega(s)$ for $s>0$,
such that
\[
|a(z_1,y)-a(z_2,y)|\le (2-\alpha) \omega(|z_1-z_2|), \quad z_1,z_2,y \in\bR^d,
\]
\[
\gamma(\tau)=\int_{|y|\le \tau}\omega(|y|)\frac{dy}{|y|^{d+\beta}}<\infty, \quad \tau>0,
\]
and, in the case of $\alpha=1$,
for  all $0<r<R$ and $x\in \bR^d$,
\[
\int_{r<|y|<R} ya(x,y)\frac{dy}{|y|^{d+\alpha}}.
\]
Also suppose that there is a positive constant $K_1$ such that
\[
\sup_{z\in B_{2\varepsilon},y\in \bR^d}|a(z,y)|\le (2-\alpha)K_1.
\]
Then there is a constant $N=N(d,p,\alpha,\beta,[w]_{A_p})$ such that for $u\in\mathcal{S}(\bR^d)$,
\begin{equation}
							\label{eq1004_01}
    \|\varphi \cL u\|_{L_{p,w}}^p\le N C(\varepsilon,\varphi, \omega,K_1)\|\partial^\alpha u\|_{L_{p,w}}^p,
\end{equation}
where
\[C(\varepsilon,\varphi, \omega,K_1)
=K_1^p\|\varphi\|_{L_\infty}^p
+\varepsilon^p K_1^p \|D\varphi\|_{L_\infty}^p
+ \varepsilon^{\beta p} \|\varphi\|_{L_\infty}^p\gamma^p(\varepsilon).
\]
The constant $N$ can be chosen so that $N = N(d,p, \alpha_0, \alpha_1, \beta,[w]_{A_p})$ if $0 < \alpha_0 \leq \alpha \leq \alpha_1 < 1$ and $N = N(d,p, \alpha_0, \beta,[w]_{A_p})$ if $1 < \alpha_0 \leq \alpha < 2$, respectively.
\end{lemma}
\begin{proof}

We first consider the case when $p\in (d/\beta,\infty)$.
To clarify the dependency of $N$ with respect to $\alpha$, we use $N_1$ when the constants in the inequalities below do not depend on $\alpha$.
For a measurable function $f$ on $\bR^d\times \bR^d$ and $r>0$, denote
\[
\mathcal{K}(r,f)=(2-\alpha)^{-1}\sup_{z\in B_r, y\in\bR^d}|f(z,y)|.
\]

Assume that  $a(x,y)$ are infinitely differentiable in $x$.
By Lemma \ref{0501_1}, we see that, for any $x\in \bR^d$,
\begin{equation*}
    \varphi(x)a(x,y)=N_1\int_{\bR^d}|x-z|^{-d+\beta}\partial^\beta  (\varphi(\cdot) a(\cdot,y))(z)\,dz.
\end{equation*}
Using this, we observe that for any $x \in \bR^d$,
\begin{align}
&\varphi(x)\cL^a u(x)=
\int_{\bR^d}\nabla_y^\alpha u(x)\varphi(x)a(x,y)\frac{dy}{|y|^{d+\alpha}}\notag\\
&= N_1\int_{\bR^d}\nabla_y^\alpha u(x) \left( \int_{\bR^d} |x-z|^{-d+\beta} \partial^\beta \big(\varphi(\cdot)a(\cdot,y)\big)(z)\,dz \right)\frac{dy}{|y|^{d+\alpha}}\notag\\
&=N_1
\int_{\bR^d}|x-z|^{-d+\beta} \left(\int_{\bR^d}\nabla_y^\alpha u(x)\partial^\beta \big(\varphi(\cdot)a(\cdot,y)\big)(z)\frac{dy}{|y|^{d+\alpha}}\right)\,dz\notag\\
    \label{05182115}
&:=N_1\int_{\bR^d} |x-z|^{-d+\beta} \cL_z^{\partial^\beta(\varphi a)}u(x)\,dz.
\end{align}

For $|x|\le \varepsilon$,
\begin{equation}\label{0501_3}
    |x-z|^{-d+\beta}\le N_1 |z|^{-d+\beta},\quad |z|\ge 2\varepsilon.
\end{equation}
Hence, from (\ref{05182115}) and (\ref{0501_3}),  we have
\begin{align*}
&\int_{\bR^d} |\varphi(x)\cL^a u(x)|^p w(x) \,dx =\int_{B_\varepsilon} |\varphi(x)\cL^a u(x)|^p w(x) \,dx\\
&\le N_1\int_{B_{\varepsilon}}\left(\int_{\bR^d}|x-z|^{-d+\beta} \left \lvert \cL_z^{\partial^\beta (\varphi a)}u(x)\right\rvert\,dz\right) ^p w(x) \,dx\\
&\le N_1 \int_{B_\varepsilon}\left(\int_{B_{2\varepsilon}}|x-z|^{-d+\beta }\left|\cL_z^{\partial^\beta (\varphi a)}u(x)\right|\,dz\right)^p w(x) \,dx\\
&\quad + N_1\int_{B_\varepsilon}\left(\int_{B_{2\varepsilon}^c}|z|^{-d+\beta }\left|\cL_z^{\partial^\beta (\varphi a)}u(x)\right|\,dz\right)^p w(x) \,dx\\
&=: J_1+J_2.
\end{align*}

\noindent
{\em Estimate of $J_1$.}
Using H\"older's inequality, we have
\[
J_1\le N_1\int_{B_{\varepsilon}} \left(\int_{B_{2\varepsilon}}|x-z|^{(-d+\beta)\frac{p}{p-1}}\,dz\right)^{p-1}\left(\int_{B_{2\varepsilon}}\left|\cL_z^{\partial^\beta (\varphi a)}u(x)\right|^p\,dz\right) w(x) \,dx.
\]
Note that, for $|x|\le \varepsilon$,
\[
\left(\int_{B_{2\varepsilon}}|x-z|^{(-d+\beta)\frac{p}{p-1}}\,dz\right)^{p-1}
\le
\left(\int_{B_{3\varepsilon}}|z|^{(-d+\beta)\frac{p}{p-1}}\,dz\right)^{p-1}
\le
N_1\varepsilon^{\beta p-d}.
\]
Hence, we have that
\begin{multline*}
J_1 \le N_1\varepsilon^{\beta p-d} \int_{B_{\varepsilon}}\int_{B_{2\varepsilon}}\left|\cL_z^{\partial^\beta (\varphi a)}u(x)\right|^p w(x) \,dz\,dx
\\
\le N_1\varepsilon^{\beta p-d} \int_{B_{2\varepsilon}}\int_{\bR^d}\left|\cL_z^{\partial^\beta (\varphi a)}u(x)\right|^p w(x) \,dx\,dz.
\end{multline*}
By Proposition \ref{05132149}, we see that
\begin{align}
J_1&\le N_1N^p \varepsilon^{\beta p-d}(2-\alpha)^{-p} \int_{B_{2\varepsilon}}\left(\sup_{y\in\bR^d}|\partial_z^\beta\big(\varphi(\cdot) a(\cdot,y)\big)(z)|^p \right)\int_{\bR^d} |\partial ^\alpha u|^p w(x) \,dx\,dz \nonumber
\\
&\le  N_1N^p \varepsilon^{\beta p}(2-\alpha)^{-p} \left(\sup_{y\in\bR^d,z\in B_{2\varepsilon}}|\partial_z^\beta\big(\varphi(\cdot) a(\cdot,y)\big)(z)|^p \right)\int_{\bR^d} |\partial ^\alpha u|^p w(x) \,dx, \label{eq0523_03}
\end{align}
where $N$ is the constant in  Proposition \ref{05132149}.
To estimate the supremum of
$$
|\partial_z^\beta\big(\varphi(\cdot) a(\cdot,y)\big)(z)|
$$
in \eqref{eq0523_03}, we observe that, for any $y,z\in \bR^d$,
\begin{align}
&|\partial_z^\beta\big(\varphi(\cdot) a(\cdot,y)\big)(z)|\notag\\
&\le N_1
\int_{|h|>\varepsilon} |\varphi(z+h)a(z+h,y)-\varphi(z)a(z,y)|\frac{dh}{|h|^{d+\beta}}\notag\\
&\quad +N_1
\int_{|h|<\varepsilon} |\varphi(z+h)a(z+h,y)-\varphi(z)a(z,y)|\frac{dh}{|h|^{d+\beta}}\notag\\
&\le N_1
\int_{|h|>\varepsilon} |\varphi(z+h)a(z+h,y)-\varphi(z)a(z,y)|\frac{dh}{|h|^{d+\beta}}\notag\\
&\quad +N_1|a(z,y)|\int_{|h|\le \varepsilon}|\varphi(z+h)-\varphi(z)|\frac{dh}{|h|^{d+\beta}}\notag\\
&\quad +N_1|\varphi(z)|\int_{|h|\le \varepsilon} |a(z+h,y)-a(z,y)|\frac{dh}{|h|^{d+\beta}}\notag\\
&\quad +N_1\int_{|h|\le \varepsilon} |a(z+h,y)-a(z,y)||\varphi(z+h)-\varphi(z)|\frac{dh}{|h|^{d+\beta}}\notag\\
                \label{05182153}
&=: I_1+I_2+I_3+I_4.
\end{align}
For any $y\in \bR^d$ and $z\in B_{2\varepsilon}$,
\begin{align*}
I_1
&\le N_1 \sup_{z,y}|\varphi(z)a(z,y)|\varepsilon^{-\beta}\\
&\le N_1\sup_{z\in B_{2\varepsilon},y\in \bR^d}|a(z,y)|\|\varphi\|_{L_\infty}\varepsilon^{-\beta}
\le N_1(2-\alpha)  \mathcal{K}(2\varepsilon,a) \|\varphi\|_{L_\infty}\varepsilon^{-\beta},
\end{align*}
\[
I_2\le N_1 \sup_{z\in B_{2\varepsilon},y\in \bR^d}|a(z,y)| \|D\varphi\|_{L_\infty}\varepsilon^{1-\beta}
\le N_1 (2-\alpha)\mathcal{K}(2\varepsilon,a)\|D\varphi\|_{L_\infty}\varepsilon^{1-\beta},
\]
\[
I_3,I_4\le N_1(2-\alpha) \|\varphi\|_{L_\infty}\gamma(\varepsilon).
\]
From the above estimates of $I_i$, $i=1,2,3,4$, together with \eqref{eq0523_03}, we obtain that
\begin{equation}
								\label{eq0523_04}
J_1 \leq N C(\varepsilon, \varphi, \omega, \mathcal{K}(2\varepsilon,a)) \|\partial^\alpha u\|_{L_{p,w}}^p,
\end{equation}
where $N$ can be chosen as described in the lemma.

\vspace{1em}

\noindent
{\em Estimate of $J_2$.}
By the Minkowski inequality
\[
J_2^{1/p} = N_1 \left\| \int_{B_{2\varepsilon}^c}|z|^{-d+\beta}\left|\cL_z^{\partial^\beta (\varphi a)}u(\cdot)\right|\,dz \right\|_{L_{p,w}(B_\varepsilon)}
\]
\[
\leq N_1 \int_{B_{2\varepsilon}^c}|z|^{-d+\beta} \left\|\cL_z^{\partial^\beta (\varphi a)}u(\cdot)\right\|_{L_{p,w}(B_\varepsilon)} \, dz
\]
\[
\leq N_1 \int_{B_{2\varepsilon}^c}|z|^{-d+\beta} \left\|\cL_z^{\partial^\beta (\varphi a)}u(\cdot)\right\|_{L_{p,w}(\bR^d)} \, dz.
\]
Then by Proposition \ref{05132149}, it follows that
\[
J_2^{1/p} \leq N_1N \|\partial^\alpha u\|_{L_{p,w}} \int_{B_{2\varepsilon}^c}|z|^{-d+\beta} \sup_{y \in \bR^d} |\partial_z^\beta \left( \varphi(\cdot) a(\cdot,y)\right)(z)| \, dz,
\]
where $N$ is the constant in  Proposition \ref{05132149}.
For $z \in B_{2\varepsilon}^c$, we see from \eqref{05182153} that
\[
I_2 = I_3 = I_4 = 0
\]
and
\[
I_1 =N_1 \int_{|h|>\varepsilon} |\varphi(z+h) a(z+h,y)| \frac{dh}{|h|^{d+\beta}},
\]
which imply that $J_2^{1/p}$ is bounded by
\[
N_1 N(2-\alpha)^{-1} \|\partial^\alpha u\|_{L_{p,w}} \int_{B_{2\varepsilon}^c} |z|^{-d+\beta} \int_{|h|> \varepsilon} \sup_{y \in \bR^d} |\varphi(z+h) a(z+h,y)| \, \frac{dh}{|h|^{d+\beta}} \, dz
\]
\[
= N_1 N(2-\alpha)^{-1} \|\partial^\alpha u\|_{L_{p,w}} \int_{|h|> \varepsilon}\int_{B_{2\varepsilon}^c} |z|^{-d+\beta}  \sup_{y \in \bR^d} |\varphi(z+h) a(z+h,y)| \, dz \, \frac{dh}{|h|^{d+\beta}},
\]
where
\[
|z|^{-d+\beta}\sup_{y \in \bR^d} |\varphi(z+h) a(z+h,y)| = 0
\]
for $z+h \in B_\varepsilon^c$ and
\[
|z|^{-d+\beta}\sup_{y \in \bR^d} |\varphi(z+h) a(z+h,y)| \leq |z|^{-d+\beta} \sup_{z \in B_\varepsilon, y \in \bR^d} |a(z,y)| \|\varphi\|_{L_\infty}
\]
for $z+h \in B_\varepsilon$.
Thus,
\begin{align*}
&\int_{B_{2\varepsilon}^c} |z|^{-d+\beta}  \sup_{y \in \bR^d} |\varphi(z+h) a(z+h,y)| \, dz\\
&\leq \int_{B_{2\varepsilon}^c \cap \{z+h \in B_\varepsilon\}} \varepsilon^{-d+\beta} \sup_{z \in B_\varepsilon, y \in \bR^d} |a(z,y)| \|\varphi\|_{L_\infty} \, dz\\
&\leq N_1 \varepsilon^\beta \sup_{z \in B_\varepsilon, y \in \bR^d} |a(z,y)| \|\varphi\|_{L_\infty}.
\end{align*}
This shows that
\begin{align*}
J_2^{1/p} &\leq N_1 N(2-\alpha)^{-1} \|\partial^\alpha u\|_{L_{p,w}} \int_{|h|>\varepsilon} \varepsilon^\beta \sup_{z \in B_\varepsilon, y \in \bR^d}|a(z,y)|\| \varphi\|_{L_\infty} \, \frac{dh}{|h|^{d+\beta}}\\
&= N_1 N(2-\alpha)^{-1} \|\partial^\alpha u\|_{L_{p,w}} \sup_{z \in B_\varepsilon, y \in \bR^d}|a(z,y)| \|\varphi\|_{L_\infty}\\
&\leq N_1 N \mathcal{K}(2\varepsilon,a) \|\partial^\alpha u\|_{L_{p,w}} \|\varphi\|_{L_\infty}.
\end{align*}

By combining the above estimate for $J_2$ and \eqref{eq0523_04} for $J_1$,  we see that the inequality \eqref{eq1004_01} holds for $a(x,y)$ which is infinitely differentiable in $x$ with $C(\varepsilon, \varphi,\omega, K_1)$ replaced with $C\left(\varepsilon, \varphi,\omega,\mathcal{K}(2\varepsilon,a)\right)$.

For general $a$, set
\[
a_n(x,y)=a(x,y)*\eta_n(x), \quad \eta(x)= n^{d} \eta(nx)
\]
for $n=1,2,\dots$, where $\eta$ is a standard mollifier on $\bR^d$.
Note that
\[
|a_n(z_1,y)-a_n(z_2,y)|\le (2-\alpha) \omega(|z_1-z_2|), \quad z_1,z_2,y \in\bR^d.
\]
By the inequality \eqref{eq1004_01} proved above for $a_n$, we have
\[
\|\varphi \cL^{a_n} u\|_{L_{p,w}}^p \leq N C\left(\varepsilon, \varphi, \omega, \mathcal{K}(2\varepsilon,a_n)\right) \|\partial^\alpha u\|_{L_{p,w}}^p.
\]
Since for any $\rho>0$,
	\[
	\mathcal{K}(2\varepsilon,a_n)\le\mathcal{K}(2\varepsilon+\rho,a)
	\]
for large $n$, by Fatou's Lemma, we see that for any $\rho>0$,
\begin{align*}
&\|\cL^{a}u\|_{L_{p,w}}^p
\le \underset{n\to \infty}{\operatorname{liminf}}
\|\cL^{a_n}u\|_{L_{p,w}}^p\\
&\le \underset{n\to \infty}{\operatorname{liminf}} NC(\varepsilon,\rho,\omega,\mathcal{K}(2\varepsilon,a_n))\|\partial^\alpha u\|_{L_{p,w}}^p
\le
NC(\varepsilon,\rho,\omega,\mathcal{K}(2\varepsilon+\rho,a))\|\partial^\alpha u\|_{L_{p,w}}^p.
\end{align*}
Upon noting that
\[
\lim_{\rho \to 0}C(\varepsilon,\rho,\omega,\mathcal{K}(2\varepsilon+\rho,a)) \leq C(\varepsilon,\rho,\omega,K_1),
\]
we finish the proof when $p > d/\beta$.

For general $p\in (1,\infty)$, we apply the Rubio de Francia extrapolation theorem. See, for instance, \cite[Theorem 2.5]{MR3812104}.
The lemma is proved.
\end{proof}


We will prove Theorem \ref{05141210} by the standard freezing coefficient argument.
 One can find the argument in the proof of \cite[Lemma 1.6.3]{MR2435520} for second-order elliptic equations, and   in the proof of \cite[Lemma 8]{MR3145767} for non-local equations.
To show the exact parameters which the constants depend on in the estimates, we give a detailed proof of the following lemma.

\begin{lemma}\label{05141249}
Assume that $\beta\in(0,1)$, $\alpha\in(0,2)$, and $p\in (1,\infty)$.
Suppose that Assumption \ref{05031403} holds.
Then, there exist $\varepsilon=\varepsilon(d,p,\alpha,\beta,\delta,K_1,\omega)$, $N=N(d,p,\alpha,\delta,K_1)$, and $\lambda_0=\lambda_0(d,p,\alpha,\beta,\delta,K_1,\omega)$ such that the following holds. If $\lambda\ge \lambda_0$, $u\in \mathcal{H}_p^{1,\alpha}(T)$ satisfies (\ref{05132019}) and for every $t\in(0,T)$, $u(t,\cdot)$ has support in a ball of radius $\varepsilon$, then
  \[
   \|u_t\|_{L_p(T)}+\|\partial^\alpha u\|_{L_p(T)}\le N \|f\|_{L_p(T)},
 \]
\[
     \|u\|_{L_p(T)}\le N (T \wedge \lambda^{-1}) \|f\|_{L_p(T)}.
\]
\end{lemma}
\begin{proof}
Without loss of generality,
we assume that for all $t\in (0,T)$, $u(t,\cdot)\in C_0^\infty(B_\varepsilon)$.
Let $\varphi\in C_0^\infty (B_{4\varepsilon})$, $0\le \varphi\le 1$, $\varphi(x)=1$ if $|x|\le 2\varepsilon$, with
$$
\|D\varphi\|_{L_\infty}\le \frac{N}{\varepsilon}.
$$
Set
\[
\cL_0u(t,x)=\int_{\bR^d}\nabla_y^\alpha u(t,x) a(t,0,y)\frac{dy}{|y|^{d+\alpha}}.
\]
Note that for any $\lambda >0$,
by Proposition 1 in \cite{MR3145767},
 there is a constant $N_1=N_1(d,p,\alpha,\delta,K_1)$ such that
\[
\|u\|_{L_p(T)}
\le
N_1 (T\wedge \lambda^{-1})\|u_t-\cL_0u+\lambda u\|_{L_p(T)}
\]
and
\[
\|u_t\|_{L_p(T)}+\|\partial^\alpha u \|_{L_p(T)}
\le
N_1\|u_t-\cL_0u+\lambda u\|_{L_p(T)}.
\]
Then by the triangle inequality, we have
\begin{align*}
\|u\|_{L_p(T)}
&\le
N_1 (T\wedge \lambda^{-1})\left(\|f\|_{L_p(T)}+\|\varphi(\cL-\cL_0)u\|_{L_p(T)}\right.\\
&\quad \left.+\|(\cL u-\cL_0u)(1-\varphi)\|_{L_p(T)}
\right)
\end{align*}
and
\begin{align*}
&\|u_t\|_{L_p(T)}+\|\partial^\alpha u \|_{L_p(T)}
\\
&\le
N_1\left(\|f\|_{L_p(T)}+\|\varphi(\cL-\cL_0)u\|_{L_p(T)}
+\|(\cL u-\cL_0u)(1-\varphi)\|_{L_p(T)}
\right).
\end{align*}
Observe that for any $t\in (0,T]$,
\[
\sup_{z\in B_{8\varepsilon},y\in \bR^d}|a(t,z,y)-a(t,0,y)|^p\le \omega(8\varepsilon)^p.
\]
Due to Lemma \ref{05021011} with $a\to a(t,x,y)-a(t,0,y)$ and $\varepsilon \to 4\varepsilon$, we see that
\[
\|\varphi(\cL-\cL_0)u\|_{L_p(T)}\le N_0 C(\varepsilon)\|\partial^\alpha u\|_{L_p(T)},
\]
where
\[
C(\varepsilon)=\omega(8\varepsilon)+\gamma(4\varepsilon)\varepsilon^{\beta},
\]
which goes to zero as $\varepsilon\to 0$.
Since $u$ has support in the ball of radius $\varepsilon$ centered at $0$,  using the Minkowski inequality, we see that for each $t\in(0,T)$,
\begin{align*}
&\|(\cL u-\cL_0u)(1-\varphi)\|_{L_p}\\
&=\left(\int_{B_{2\varepsilon}^c} \left \lvert
\int_{|y|>\varepsilon} \nabla^\alpha_y u(t,x)(a(t,x,y)-a(t,0,y))\frac{dy}{|y|^{d+\alpha}}\right \rvert^p |1-\varphi(x)|^p \,dx\right)^{1/p}\\
&=\left(\int_{B_{2\varepsilon}^c} \left \lvert
\int_{|y|>\varepsilon} u(t,x+y)(a(t,x,y)-a(t,0,y))\frac{dy}{|y|^{d+\alpha}}\right \rvert^p  \,dx\right)^{1/p}\\
&\le
N\int_{|y|>\varepsilon}\|u(t,\cdot+y)\|_{L_p}\frac{dy}{|y|^{d+\alpha}}
\le N_2\varepsilon^{-\alpha}\|u(t,\cdot)\|_{L_p}.
\end{align*}
For a fixed $\varepsilon>0$ satisfying $N_1N_0C(\varepsilon)\le 1/2$,
we have that
\[
\|u_t\|_{L_p(T)}+\|\partial^\alpha u\|_{L_p(T)}\le 2 N_1 \|f\|_{L_p(T)}+2N_1N_2\varepsilon^{-\alpha}\|u\|_{L_p(T)}
\]
and
\[
\|u\|_{L_p(T)}\le 2N_1 (T\wedge \lambda^{-1}) \|f\|_{L_p(T)}+2N_1N_2 (T\wedge \lambda^{-1}) \varepsilon^{-\alpha}\|u\|_{L_p(T)}.
\]
Taking $\lambda_0=4N_1N_2\varepsilon^{-\alpha}$, we obtain the desired estimates.
\end{proof}

\subsection*{Proof of Theorem \ref{05140812}}
Here, we  use a version of the partition of unity argument in the proof of Theorem 1.6.4 in \cite{MR2435520}.

We may assume that $u\in\mathcal{S}$ as $\mathcal{S}$ is dense in $H_{p,w}^\alpha$.
Let
\[
\cL_1^a u(x)= \int_{|y|\le 1} (u(x+y)-u(x)-\chi_\alpha(y)(\nabla u(x),y))a(x,y)\frac{dy}{|y|^{d+\alpha}}.
\]
According to Lemmas \ref{05142036}, \ref{lem10092043}, and \ref{05191046}, we obtain
\[
 \|(\cL^a-\cL_1^a)u\|_{L_{p,w}}\le N\|u\|_{H_{p,w}^\alpha}.
\]
Therefore, it is sufficient to show that
\[
\|\cL_1^a u\|_{L_{p,w}} \le N \|u\|_{H_{p,w}^\alpha}.
\]
Let $\varphi\in C_0^\infty(B_1)$, $0\le\varphi\le1$, and $\varphi(x)=1$ if $|x|\le 1/2$.
Then for any $x\in \bR^d$,
\[
\int_{\bR^d}\left| \cL_1^a u(x) \right|^p w(x) \,dx
=N\int_{\bR^d}\int_{\bR^d}\left|\varphi(x-z)\cL_1^au(x)\right|^p w(x) \,dz\,dx.
\]
Observe that
\begin{align*}
&\varphi(x-z)\cL_1^a u(x)=\cL_1^a\big(u(\cdot)\varphi(\cdot-z)\big)(x)-u(x)\cL_1^a\big(\varphi(\cdot-z)\big)(x)\\
&\quad -\int_{|y|\le 1} \big(u(x+y)-u(x)\big)\big(\varphi(x+y-z)-\varphi(x-z)\big) a(x,y)\frac{dy}{|y|^{d+\alpha}}\\
&=:I_1+I_2+I_3.
\end{align*}
We first estimate $I_2$.  Using the Minkowski inequality in $z$ and Lemma \ref{05151257} to $\varphi$ without weights, we see that
\[
\int\int I_2^p w(x) \,dz \,dx\le N \int |u(x)|^p \|\varphi\|_{H_p^2}^p w(x) \,dx \le N \|u\|_{L_{p,w}}^p.
\]
For $I_3$, using Lemma \ref{05140940}  and the fact that $|a|\le (2-\alpha)K_1$, we have
\[
\int\int I_3^p w(x) \,dz\,dx\le N\|u\|_{H_{p,w}^\alpha}.
\]
To estimate $I_1$, we observe that
\[
\cL_1^a\big(u(\cdot)\varphi\left(\cdot-z\right)\big)(x)=\varphi\left(\frac{x-z}{4}\right)\cL_1^a\big(u(\cdot)\varphi(\cdot-z)\big)(x).
\]
Thus by  Lemma \ref{05021011}, we see that
\begin{align*}
&\int_{\bR^d}\int_{\bR^d}\left|\cL_1^a
\big(u(\cdot)\varphi(\cdot-z)\big)(x)\right|^p w(x) \,dx\,dz\\
&\le N \int_{\bR^d} \int_{\bR^d}\left|\partial^\alpha \big(u(\cdot)\varphi(\cdot-z)\big)(x)\right|^p w(x) \,dx \,dz\\
&\le N \int_{\bR^d} \int_{\bR^d}\left|\partial^\alpha u(x) \varphi(x-z)\right|^p w(x) \,dx \,dz\\
&\quad + N \int_{\bR^d} \int_{\bR^d}\left|\partial^\alpha\big(\varphi(\cdot-z)\big)(x)u(x)\right|^p w(x) \,dx \,dz\\
&\quad+ N
\int_{\bR^d}\int_{\bR^d}\left(\int_{\bR^d}|u(x+y)-u(x)||\varphi(x+y-z)-\varphi(x-z)| \frac{dy}{|y|^{d+\alpha}}\right) w(x) \,dz\,dx\\
&=:J_1+J_2+J_3.
\end{align*}
By the same calculations as in the estimates of $I_2$ and $I_3$,  we derive that
\[
J_2+J_3\le N \|u\|_{H_{p,w}^\alpha}^p.
\]
Note that $J_1=N\|\partial^\alpha u\|_{L_{p,w}}$.
Consequently, by Lemma \ref{lem10092043} we reach the result.\qed

\subsection*{Proof of Theorem \ref{05141210}}
Using the partition of unity argument and standard freezing coefficient technique with Lemma \ref{05141249}, we have
  (\ref{05132016}) and (\ref{05132020}).
  The uniqueness comes from the estimates.
  Finally, using Theorem \ref{05140812} and the  method of continuity, we reach the existence of solutions.\qed

\subsection*{Proof of Theorem \ref{thm05241}}

 Temporarily, we assume that $\alpha>1$.
 Fix $\lambda\ge \lambda_0$, where $\lambda_0$ is the constant in the statement of Theorem \ref{05141210}.
 Let $\tau\in(0,T).$
 Due to Theorem \ref{05141210}, we have
\begin{align*}
& \|u\|_{\cH_p^{1,\alpha}(\tau)}
\le N\|u_t-\cL^au+\lambda u\|_{L_p(\tau)}\\
& \le N\|u_t-\cL^au+b^iD_i u+cu\|_{L_p(\tau)}
 + N\|u\|_{L_p(\tau)}+N\|Du\|_{L_p(\tau)},
\end{align*}
where the constant $N$ is independent of $\tau$.
  By an interpolation inequality with respect to the spatial variables,
 \[
 \|Du\|_{L_p(\tau)}\le \varepsilon\|\partial^\alpha u\|_{L_p(\tau)}+N(\varepsilon)\|u\|_{L_p(\tau)}.
  \]
 By taking sufficiently small $\varepsilon$, we see that
\begin{align*}
 \|u\|_{\cH_p^{1,\alpha}(\tau)}
  &\le N\|u_t-\cL^au+b^iD_i u+cu\|_{L_p(\tau)}
  + N\|u\|_{L_p(\tau)}\\
&  =N\|f\|_{L_p(\tau)}
  + N\|u\|_{L_p(\tau)}.
\end{align*}
  Observe that
\begin{align*}
&  \|u(t,\cdot)\|_{L_p(\bR^d)}^p
  \le \left\|\int_0^t |u_t|(s,\cdot)\,ds \right\|_{L_p(\bR^d)}^p\\
&   \le t^{p-1}\left(\int_0^t \|u_t(s,\cdot)\|_{L_p(\bR^d)}^p \,ds\right)
  =t^{p-1}\|u_t\|_{L_p(t)}^p
\end{align*}
for any $t\in (0,T)$.
Hence, by taking integration over $(0,\tau)$, we have \[
  \|u\|_{L_p(\tau)}^p\le
  \left(\int_0^\tau t^{p-1}\|u_t\|_{L_p(t)}^p \,dt \right)
  \le \left(\int_0^\tau t^{p-1}\|u\|_{\cH_p^\alpha(t)}^p \,dt \right).
  \]
  Thus,
  \[
  \|u\|_{\cH_p^{1,\alpha}(\tau)}^p
  \le
  N\|f\|_{L_p(\tau)}^p
  + N\left(\int_0^\tau t^{p-1}\|u\|_{\cH_p^\alpha(t)}^p \,dt \right).
  \]
 Set
 \[
 \cA(s)=\|u\|_{\cH_p^{1,\alpha}(s)}^p,
 \quad
 \cB(s)=N\|f\|_{L_p(s)}^p.
 \quad
 \cC(s)= Ns^{p-1}
 \]
 for $s\in (0,T)$.
 Then, we see that
 \[
 \cA(\tau)\le \cB(\tau)+\int_0^\tau \cC(t) \cA(t)\,dt
 \]
 for any $\tau\in(0,T).$
 By Gr\"onwall's inequality,
 \[
 \cA(T)\le \cB(T) e^{\int_0^T \cC(t)\,dt}.
 \]
 This proves the theorem.\qed

\bibliographystyle{plain}


 \def\cprime{$'$}


\end{document}